\documentclass[11pt, reqno]{amsart}
\usepackage{amssymb, amsthm, amsmath, amsfonts}
\usepackage{graphics}
\usepackage{hyperref}
\usepackage{bbm}
\usepackage[all]{xy}
\usepackage{enumerate}
\usepackage[mathscr]{eucal}
\usepackage{enumerate}
\usepackage{thmtools}
\usepackage{xcolor}
\usepackage{ytableau}
\usepackage{tikz-cd}
\usepackage{cite}
\usepackage{extarrows}
\usepackage[T1]{fontenc}

\theoremstyle{plain}
\newtheorem{theorem}{Theorem}[section]
\newtheorem{lemma}[theorem]{Lemma}
\newtheorem{proposition}[theorem]{Proposition}

\newtheorem{corollary}[theorem]{Corollary}
\numberwithin{equation}{section}

\theoremstyle{definition}

\newtheorem{definition}[theorem]{Definition}
\newtheorem{example}[theorem]{Example}
\newtheorem{remark}[theorem]{Remark}

\newcommand{\Mod}{{\textrm{-}\mathrm{Mod}}}

\DeclareMathOperator{\Ob}{Ob}
\DeclareMathOperator{\Aut}{Aut}

\DeclareMathOperator{\Hom}{Hom}
\DeclareMathOperator{\End}{End}

\DeclareMathOperator{\Ext}{Ext}

\DeclareMathOperator{\Ima}{Im}
\DeclareMathOperator{\id}{id}

\DeclareMathOperator{\rank}{rank}

\DeclareMathOperator{\Gr}{Gr}

\DeclareMathOperator{\dep}{dep}

\newcommand{\FI}{{\mathrm{FI}}}
\newcommand{\VA}{{\mathrm{VA}}}
\newcommand{\C}{{\mathscr{C}}}
\newcommand{\D}{{\mathscr{D}}}

\newcommand{\T}{{\mathscr{T}}}

\newcommand{\op}{{\mathrm{op}}}

\title{Standard modules and Dold--Kan correspondences for generalized Reedy categories}

\author{Liping Li}
\address{School of Mathematics and Statistics, Hunan Normal University, Changsha 410081, China.}
\email{lipingli@hunnu.edu.cn}
\thanks{The author is partly supported by NSFC Grant No. 12571037.}


\begin{document}

\begin{abstract}
We develop a representation-theoretic approach to generalized Reedy categories through a systematic study of standard modules. For a broad class of these categories, we provide a uniform, computable framework that reduces abstract Hom-spaces between standard modules to elementary linear algebra and graph theory via incidence matrices and morphism fibers. As a primary application, we establish a uniform extension of the Dold--Kan correspondence for certain categories arising from rooted trees, and obtain an orthogonal decomposition for finite spiders, including the category of finite sets and partial injections. Crucially, this machinery unifies and provides a single conceptual framework for several classical, seemingly disparate results across algebraic topology and representation theory, including Kuhn's decomposition theorem for vector spaces and the Thévenaz--Webb semisimplicity theorem for Mackey functors.
\end{abstract}

\maketitle

\section{Introduction}

\subsection{Motivation}

Reedy categories occupy a central position in modern homotopy theory for the study of simplicial and cosimplicial objects, homotopy limits and colimits, and higher categorical constructions. In particular, the existence of Reedy model structures provides a powerful framework for studying diagrams in model categories. Explicitly, if $\mathcal C$ is a Reedy category and $\mathcal M$ is a model category, then the diagram category $\mathcal M^{\mathcal C}$ admits a canonical model structure whose weak equivalences are defined objectwise; see, for instance \cite{FL91, GoerssJardine, Hirschhorn, KaygunKaya, Quillen}. The subsequent introduction of generalized Reedy categories by Berger-Moerdijk \cite{BergerMoerdijk2011} greatly expanded the scope of the theory, by encompassing a wide range of combinatorial categories arising in algebra, topology, and combinatorics. For example, categories associated to homogeneous relational structures provide a rich family of generalized Reedy categories \cite{Li}.

Alongside these homotopical developments, the representation theory of Reedy and generalized Reedy categories has recently emerged as an active area of research. Besides the classical simplex category, representations of many other interesting combinatorial generalized Reedy categories have been explored extensively, including the category of finite sets, the category of finite sets and partial injections, the category of finite-dimensional vector spaces over a finite field, and span categories of orbit categories of finite groups. Works described in \cite{CEF, DK85, GLW, Helmstutler, HM, KaygunKaya, Kuhn, LS, Pirashvili2000, PowellVespa2023, Powell, Slo, St, TW, Webb2, WGorden} have revealed remarkable phenomena such as noetherianity, decomposition theorems and semisimplicity, and Dold-Kan correspondences, suggesting the existence of a rich and unified representation theory underlying generalized Reedy categories.

Given a commutative ring $k$, a fundamental observation underlying the representation theory of generalized Reedy categories is that the $k$-linearization of a generalized Reedy category often behaves as an infinite-dimensional quasi-hereditary algebra, or more generally a standardly stratified algebra \cite{DS, DLL}. Consequently, one can define standard modules and develop a highest-weight-theoretic perspective. These modules play a distinguished role in the representation theory of generalized Reedy categories. In many important examples, every representable module admits a finite filtration by standard modules, while the simple tops of standard modules exhaust all irreducible modules up to isomorphism \cite{DS, DLL}. Thus standard modules form the basic building blocks of the representation theory of generalized Reedy categories.

A natural question is whether homomorphism spaces between standard modules can be described in a uniform and computable manner. These spaces often encode surprisingly deep structural information about the underlying generalized Reedy category. For example, in the simplex category, nontrivial Hom-spaces between distinct standard modules occur only in adjacent degrees and are generated by the alternating sums of face maps. These morphisms give rise precisely to the differentials in normalized chain complexes and thereby recover the classical Dold--Kan correspondence. More recently, in \cite{Li}, the author explicitly computed Hom-spaces between standard modules for combinatorial categories associated with five highly homogeneous relational structures classified by Cameron \cite{Ca90}. As applications, irreducible representations were classified and several Dold--Kan type correspondences were established. Existing examples therefore suggest that Hom-spaces between standard modules constitute a fundamental invariant of generalized Reedy categories. They appear to govern a wide range of phenomena, including decomposition theorems \cite{Kuhn, St, TW}, semisimplicity results \cite{TW, Webb2}, Dold--Kan type correspondences \cite{GLW, KaygunKaya, LS, Pirashvili2000, Slo}, and other structural properties of representation categories.

The purpose of this paper is to develop a systematic theory of Hom-spaces between standard modules of generalized Reedy categories. We provide explicit combinatorial descriptions of these spaces in terms of fibers of morphisms, incidence matrices, and associated graphs, thereby reducing their computation to elementary linear algebra. This approach reveals unexpected connections between representation theory, combinatorics, and spectral graph theory. As applications, we obtain a conceptual explanation for several decomposition and semisimplicity phenomena appearing in the literature, establish a uniform extension of the Dold--Kan correspondence, and introduce a new family of generalized Reedy categories arising from finite rooted trees.

\subsection{Main results and applications}

We now describe the main results of the paper and some of their applications.

\subsubsection{Hom-spaces between standard modules}

Let $\C$ be a locally finite generalized Reedy category equipped with a degree function
\[
d: \Ob(\C) \to \mathbb N.
\]
For every pair of objects $x, y \in \Ob(\C)$, we associate a finite incidence matrix $M_{x,y}$ constructed from fibers of unfactorizable morphisms in the inverse category $\C^{-}$. The precise construction is given in Section \ref{combinatorial construction}, but the essential point is that $M_{x,y}$ depends only on the combinatorics of the Reedy structure.

Our first main theorem identifies Hom-spaces between standard modules with kernels of these matrices. Here the standard module is defined intrinsically as the quotient of a representable module by morphisms with noninvertible negative Reedy factors; see Section 3 for details.

\begin{theorem} [Theorem \ref{hom spaces general}] \label{thm1}
Let $\C$ be a generalized Reedy category satisfying the standing finiteness assumptions and $k$ a commutative ring. Then for every pair of objects $x, y \in \Ob(\C)$ with $d(y)>d(x)$ there is a natural isomorphism
\[
\Hom_{k\C}(\Delta_y,\Delta_x) \cong \ker(M_{x,y}^{\mathsf T}).
\]
\end{theorem}

The incidence matrix $M_{x,y}$ and its associated Gram matrix admit combinatorial interpretations in terms of fibers and weighted graphs. Consequently, Hom-spaces between standard modules can be analyzed using elementary linear algebra, graph theory, and spectral methods.

\subsubsection{Decomposition theorems}

A remarkable phenomenon occurring in several generalized Reedy categories is the existence of decomposition equivalences of the form
\[
\C\Mod \simeq \prod_{x \in \Ob(\C)} kG_x \Mod,
\]
where $\C \Mod$ is the category of all $\C$-modules (covariant functors from $\C$ to $k \Mod$), and $G_x = \Aut_{\C}(x)$. Examples include decomposition theorems for span categories, the semisimplicity theorem for Mackey functors by Th\'{e}venaz--Webb, and Kuhn's decomposition theorem for categories of finite vector spaces over finite fields. For details, see \cite{CEF, DLL, Kuhn, Slo, St}.

Our next theorem gives a necessary and sufficient condition for such decomposition equivalences.

\begin{theorem} [Theorem \ref{decomposition}] \label{thm 2}
Let $k$ be a field and let $\C$ be a generalized Reedy category satisfying a mild projectivity condition. Then the functor
\[
\bigl(\Hom_{k\C}(\Delta_x, -)\bigr)_{x \in\Ob(\C)}
\]
induces an equivalence of categories
\[
\C\Mod \simeq \prod_{x \in \Ob(\C)} kG_x \Mod
\]
if and only if
\begin{enumerate}
\item $\Hom_{k\C}(\Delta_y, \Delta_x) = 0$ whenever $x$ and $y$ are not isomorphic;

\item $|\C^+(x,y)| = |\C^-(y,x)|$ for all $x, y \in \Ob(\C)$.
\end{enumerate}
\end{theorem}

\begin{remark}
Theorem~\ref{thm 2} addresses the question studied by Street in \cite{St}: when does the core groupoid of a category suffice to determine its representation category? Street establishes such an equivalence for categories equipped with a proper factorization system, under an idempotent axiom and a finiteness condition on subobjects. In situations where both approaches apply, his construction can be compared with the standard modules considered here. The distinction lies in the hypotheses and the form of the criterion: under our standing projectivity assumption, Theorem~\ref{thm 2} gives necessary and sufficient conditions in terms of Hom-orthogonality of standard modules and a numerical symmetry between positive and negative morphisms. Thus it provides a representation-theoretic characterization of when the core groupoid suffices in the present generalized Reedy setting, whereas Street's theorem supplies sufficient conditions in a broader categorical framework and allows coefficients in a commutative ring.
\end{remark}

The above theorem provides a common explanation for several seemingly unrelated results in the literature. For instance:
\begin{itemize}
\item it recovers decomposition results for span categories in \cite{DLL}, and in particular, the decomposition theorem for finite sets and partial injections by Church-Ellenberg-Farb \cite{CEF};

\item with the same philosophy, we can recover the semisimplicity theorem for Mackey functors by Th\'{e}venaz--Webb \cite{TW};

\item in characteristic zero, it interprets Kuhn\'s decomposition theorem \cite{Kuhn} in terms of the spectrum of naturally associated weighted graphs.
\end{itemize}

\subsubsection{A uniform Dold--Kan correspondence}

Our second main application is a uniform extension of the Dold--Kan correspondence of the simplex category. To establish this result we introduce a new generalized Reedy category $\T$ whose objects are finite rooted trees and whose morphisms are generated by embeddings and admissible elementary contractions. We prove that $\T$ admits a natural generalized Reedy structure and that its standard modules are also projective.

\begin{theorem} [Proposition \ref{poset is Reedy} and Theorem \ref{thm:projective_standard}] \label{thm 3}
The category $\T$ admits a generalized Reedy structure. Moreover, its standard modules are projective and form a family of projective generators.
\end{theorem}

The simplex category embeds naturally into $\T$, so this theorem already recovers the classical projectivity of simplicial standard modules. The next theorem provides the analogue of the Dold--Kan correspondence itself.

\begin{theorem} [Theorem \ref{generalized D-K}] \label{thm 4}
Let $k$ be a commutative ring and $\D$ a subcategory of $\T$ satisfying Conditions (C0)--(C2) of Subsection \ref{uniform D-K correspondence}. Then there exists a $k$-linear category $\mathscr E$ with
\[
\D \Mod \simeq \mathscr{E} \Mod
\]
such that $\mathscr{E}(A, A) \cong k\Aut(A)$ for every object $A$, and
\[
\mathscr E(B, A) \neq 0 \quad \Longrightarrow \quad |A| = |B| \text{ or } |A| = |B| + 1.
\]
\end{theorem}

Thus nontrivial morphisms occur only in adjacent degrees, exactly as in the classical Dold--Kan correspondence. Furthermore, when $|A| = |B| + 1$,
\[
\mathscr E(B, A) = \Hom_{k\D} (\Delta_A, \Delta_B)
\]
is a free $k$-module whose rank is determined by a graph-theoretic invariant arising from the rooted-tree structure. For details, see Theorem \ref{generalized D-K}.

Consequently, a variety of previously unrelated Dold--Kan type phenomena arise from a common representation-theoretic mechanism. In particular, Theorem \ref{thm 4} unifies:
\begin{itemize}
\item the classical Dold--Kan correspondence for the simplex category;

\item the Dold--Kan type correspondence for the type-$D$ rooted-tree category, whose morphisms are signed order-preserving maps of finite chains.
\end{itemize}

We also consider the category $\mathscr S$ of finite spiders with root-preserving embeddings and admissible contractions. Although this category does not in general satisfy Condition (C2), its standard modules form a Hom-orthogonal family of projective generators, yielding an orthogonal decomposition of $\mathscr{S} \Mod$; see Corollary \ref{spider}. In particular, its full subcategory of finite stars recovers the classical decomposition for finite sets and partial injections.

\begin{remark}
Our category $\T$ is different from the category of trees appearing in \cite{HM, GLW}, which was introduced to study dendroidal sets and dendroidal abelian groups. In the dendroidal category $\Omega$, a rooted tree determines a free colored operad, and its morphisms are morphisms of the associated operads. By contrast, our category $\T$ regards rooted trees as posets: \emph{all} order embeddings lie in $\T^+$, while $\T^-$ is generated by admissible elementary contractions. These notions of morphism differ even for corollas. As shown in Example~\ref{ex:tree-dendroidal}, there is an order embedding between the rooted posets associated with the binary and ternary corollas, but no corresponding morphism between these corollas in $\Omega$. Consequently, the Dold--Kan type equivalences obtained here are distinct from the dendroidal Dold--Kan correspondence established in \cite{GLW}, since the underlying categories and their Reedy structures encode different combinatorial information.
\end{remark}

Taken together, these results show that Hom-spaces between standard modules provide a natural bridge between the combinatorics of generalized Reedy categories and their representation theory. They control decomposition phenomena and Dold--Kan type correspondences, while at the same time admitting explicit descriptions in terms of matrices and graphs.

\subsection{Organization}

The paper is organized as follows. In Section 2 we introduce the combinatorial and linear-algebraic constructions associated with a generalized Reedy category. These constructions are used in Section 3 to establish several equivalent descriptions of Hom-spaces between standard modules, including formulations in terms of incidence matrices, weighted graphs, and spectral data. In Section 4 we prove a characterization of decomposition equivalences in terms of Hom-orthogonality of standard modules and a cardinality symmetry condition, and consider several applications, including: decomposition theorems for span categories, the semisimplicity theorem for Mackey functors by Th\'{e}venaz--Webb, and Kuhn's decomposition theorem. Finally, in Section 5 we introduce a generalized Reedy category of finite rooted trees and use the preceding machinery to establish a uniform Dold--Kan correspondence, recovering several classical results and producing new ones.

\section{Combinatorial and linear algebraic constructions} \label{combinatorial construction}

Let $\C$ be a skeletal small category.\footnote{Many categories considered in this paper are essentially small rather than small. In such cases we replace the category by a fixed skeletal full subcategory.} The following definition is a slight variant of the notion introduced by Berger and Moerdijk \cite[Definition 1.1]{BergerMoerdijk2011}. Unlike their definition, we do not impose the condition that an automorphism acting trivially on a morphism must be the identity. This mild relaxation allows several additional examples while preserving all properties needed in this paper.

\begin{definition} \label{def of Reedy}
A small skeletal category $\C$ is called a \emph{generalized Reedy category} equipped with a degree map $d: \Ob(\C) \to \lambda$, where $\lambda$ is an ordinal, if it satisfies the following conditions:
\begin{enumerate}
\item $\C$ has a wide subcategory $\C^+$ such that $\C^+(x,y) \neq \varnothing$ only if $x=y$ or $d(x)<d(y)$;

\item $\C$ has another wide subcategory $\C^-$ such that $\C^-(x,y) \neq \varnothing$ only if $x=y$ or $d(x)>d(y)$;

\item for every $x \in \Ob(\C)$, one has $\C^+(x,x) = \C^-(x,x) = G_x$, the automorphism group of $x$;

\item every morphism $f$ in $\C$ admits a factorization $f = f^+ \circ f^-$ with $f^+$ in $\C^+$ and $f^-$ in $\C^-$, and this factorization is unique up to automorphisms.
\end{enumerate}
\end{definition}

In this paper we assume that $\lambda = \omega$, the first infinite ordinal. For the remainder of the paper we impose the following finiteness conditions:
\begin{enumerate}
\item[(a)] for each $n \in \mathbb{Z}_+$, there are only finitely many objects $x$ such that $d(x) = n$;

\item[(b)] for any pair of objects $x$ and $y$, the morphism set $\C(x, y)$ is finite.
\end{enumerate}

Note that both $\C^+$ and $\C^-$ are \textit{EI categories}, namely every endomorphism is an isomorphism, and actually an automorphism since $\C$ is skeletal. Thus we can introduce the following notion firstly appearing in \cite{LiHereditary}.

\begin{definition}
Let $\D$ be an EI category. A morphism $f$ in $\D$ is called \emph{unfactorizable} if:
\begin{enumerate}
    \item $f$ is not an isomorphism, and
    \item whenever $f = g \circ h$ in $\D$, at least one of $g$ or $h$ is an isomorphism.
\end{enumerate}
\end{definition}

The following lemma tells us that morphisms in $\C^+$ or $\C^-$ are generated by automorphisms and unfactorizable morphisms.

\begin{lemma} \label{unfactorizable}
Let $\C^+$ and $\C^-$ be as above. Then every non-invertible morphism $f$ in $\C^+$ or $\C^-$ can be written as a finite composite
\[
f = f_r \circ f_{r-1} \circ \dots \circ f_1,
\]
where each $f_i$ is an unfactorizable morphism.
\end{lemma}

\begin{proof}
We only prove the existence of factorizations for $\C^+$ since the same argument also applies to $\C^-$. Let $f: x \to y$ be a non-invertible morphism in $\C^+$. Then $d(y) > d(x)$. If $f$ is unfactorizable, we are done. Otherwise, there exists a factorization $f = g \circ h$ where both $g$ and $h$ are non-invertible, and their degrees satisfy $d(y) > d(z) > d(x)$ for the intermediate object $z$. Since $d(y)$ is finite, this process terminates in finitely many steps, producing a decomposition into unfactorizable morphisms.
\end{proof}

It is easy to check that the composite of an unfactorizable morphism and an automorphism is unfactorizable as well. In particular, given an unfactorizable morphism $f \in \C^-(y, x)$, every morphism in the orbit $G_x \cdot f$ is also unfactorizable. We choose a representative from each orbit to get a set $\mathbb{U}_{y, x} = \{h_i\}_{i \in I}$ of representative unfactorizable morphisms in $\C^-(y, x)$, and define
\[
\mathbb{U}_y = \bigsqcup_{x \in \Ob(\C)} \mathbb{U}_{y, x},
\]
which is a finite set since there are only finitely many objects $x$ such that $\C^-(y, x) \neq \varnothing$.

Every unfactorizable morphism $h \in \C^-(y, \bullet)$ with codomain $z$ defines a pullback operator
\[
h^{\ast}: \C^+(x, z) \longrightarrow \mathcal{P}(\C^+(x, y)),
\]
the power set of $\C^+(x, y)$, by
\[
\alpha \mapsto h^{\ast}(\alpha) = \{ \beta \in \C^+(x, y) \mid h \circ \beta = \alpha \}.
\]
Denote $h^{\ast}(\alpha)$ by $F_{h, \alpha}$, and call it a \emph{fiber} in $\C^+(x, y)$. Define
\[
\mathcal{F}_{x, y} = \bigl(F_{h, \alpha} \bigr)_{\substack{h \in \mathbb{U}_y, \; \alpha \in\C^+(x, \operatorname{cod}(h)) \\ F_{h, \alpha} \neq \varnothing}}.
\]
We remind the reader that $\mathcal{F}_{x,y}$ is an indexed family: fibers arising from different pairs $(h,\alpha)$ are counted separately, even if they have the same underlying subset of $\C^+(x,y)$.

Although we only use representative unfactorizable morphisms $h \in \mathbb{U}_y$, this set contains all nonempty fibers in $\C^+(x, y)$. Indeed, given a nonempty fiber $F_{h', \alpha}$ with $h' \in \C^-(y, z)$ an unfactorizable morphism, we can find an automorphism $\sigma \in G_z$ and an element $h \in \mathbb{U}_y$ such that $h' = \sigma \circ h$. Then
\[
F_{h', \alpha} = \{ \beta \in \C^+(x, y) \mid h' \circ \beta = \alpha \} = \{ \beta \in \C^+(x, y) \mid \sigma \circ h \circ \beta = \alpha \} = F_{h, \sigma^{-1} \circ \alpha}
\]
is also contained in $\mathcal{F}_{x, y}$.

Given distinct objects $x$ and $y$, we define an integer matrix $M_{x, y}$ whose rows are indexed by morphisms $f \in \C^+(x, y)$, and whose columns are indexed by fibers $F \in \mathcal{F}_{x, y}$. The entry $M_{x,y} (f, F)$ of $M_{x, y}$ at position $(f, F)$ is
\[
\delta_{f,F} =
\begin{cases}
1 & \text{if } f \in F, \\
0 & \text{otherwise}.
\end{cases}
\]
Thus each column of $M_{x,y}$ records a fiber sum $\sum_{f\in F}f$. Define the Gram matrix $T_{x, y} = M_{x, y} M^{\mathsf T}_{x, y}$, a square matrix indexed by morphisms in $\C^+(x, y)$. We regard these matrices as integer matrices and use the same notation after extension of scalars to a commutative ring $k$. The entries of this matrix admit the following combinatorial interpretation.

\begin{proposition} \label{adjancency matrix}
Let $N(f,g)$ be the entry in $T_{x, y}$ at position $(f, g)$ with $f, g \in \C^+(x, y)$. Then
\[
N(f,g) = \#\{\, F \in \mathcal{F}_{x, y} \mid f,g \in F \,\}.
\]
In particular, $N(f,f)$ is the number of fibers containing $f$.
\end{proposition}

\begin{proof}
For $f,g \in \C^+(x, y)$,
\[
N(f, g) = (M_{x,y} M_{x, y}^{\mathsf T}) (f,g) = \sum_{F \in \mathcal{F}_{x, y}} M_{x,y} (f, F) M_{x, y} (g, F).
\]
By definition of $M_{x, y}$, the product $M_{x, y} (f, F) M_{x, y} (g, F)$ equals $1$ if both $f$ and $g$ lie in $F$, and $0$ otherwise. Therefore,
\[
 N(f,g) = \#\{\, F \in \mathcal{F}_{x, y} \mid f,g \in F \,\}.
\]
\end{proof}

The above construction can also be encoded in a multigraph, called the \emph{fiber graph} $\Gamma_{x, y}$. Explicitly,
\begin{itemize}
\item The vertex set is
\[
\bigcup_{F \in \mathcal{F}_{x, y}} F;
\]
consisting precisely of those morphisms in $\C^+(x, y)$ that appear in some fiber. We retain singleton fibers as marked vertices, recording their labels $(h,\alpha)$. Morphisms belonging to no fiber are omitted from the graph but remain among the row indices of the incidence matrix above.

\item Given distinct vertices $f$ and $g$, for each indexed fiber $F_{h,\alpha}$ containing both $f$ and $g$, we draw an edge between $f$ and $g$, labelled by $(h,\alpha)$. Thus parallel edges are allowed.
\end{itemize}
This construction is independent of the choice of representative unfactorizable morphisms up to relabelling: if $h' = \sigma \circ h$ for some automorphism $\sigma$, then $h' \circ f = h' \circ g$ if and only if $h \circ f = h \circ g$.

For each fixed $h \in \mathbb{U}_y$, denote by $\Gamma_{x, y}^h$ the subgraph of $\Gamma_{x, y}$ whose vertices are morphisms $f \in \C^+(x, y)$ such that $h \circ f$ is a morphism in $\C^+$, and whose edges are those labelled by $(h,\alpha)$ for some $\alpha$. Then the set of vertices in a connected component of $\Gamma_{x, y}^h$ is precisely a fiber $F_{h, \alpha} \in \mathcal{F}_{x, y}$. Indeed, by the construction any two distinct elements in $F_{h, \alpha}$ are connected by an edge in $\Gamma_{x, y}^h$, so they lie in the same connected component. Conversely, if $f$ and $g$ lie in the same connected component, then we can find a finite sequence $f_0 = f, f_1, \ldots, f_n = g$ such that $f_i$ and $f_{i+1}$ are connected by an edge. But this means that $h \circ f_i = h \circ f_{i+1}$, so $h \circ f  = h \circ g$, and hence $f$ and $g$ are contained in the same fiber.

From the above argument we know that every connected component of $\Gamma_{x, y}^h$ is a complete graph, so $\Gamma_{x, y}$ is obtained by gluing complete graphs along common vertices, retaining parallel edges and marked singleton fibers.

By setting all diagonal entries of $T_{x, y}$ to 0, we get a matrix $\tilde{A}_{x, y}$, and call it the \emph{weighted adjacency matrix} of $\C^+(x, y)$. Its off-diagonal entry at $(f,g)$ is the number of edges between $f$ and $g$ in the fiber multigraph. Note that a row indexed by $f \in \C^+(x, y)$ has only 0 as entries in $T_{x, y}$ as an integer matrix if and only if $f$ is not a vertex in $\Gamma_{x, y}$, namely it is not contained in any fiber $F \in \mathcal{F}_{x, y}$. Consequently, after restricting $\tilde{A}_{x, y}$ to those indices corresponding to vertices of $\Gamma_{x, y}$, we obtain the \emph{reduced weighted adjacency matrix} $A_{x, y}$ of $\Gamma_{x, y}$.

\begin{lemma}\label{binary fiber criterion}
Suppose that every fiber in $\mathcal F_{x,y}$ has cardinality at most $2$. If every morphism in $\C^+(x, y)$ belongs to a fiber, and every connected component of $\Gamma_{x, y}$ contains a vertex belonging to a singleton fiber, then $M_{x,y}^{\mathsf T}$ is injective over any commutative ring $k$.
\end{lemma}

\begin{proof}
Let $(c_f)_{f \in\C^+(x,y)} \in \ker(M_{x,y}^{\mathsf T})$. For every fiber $F$ one has
\[
\sum_{f \in F}c_f=0.
\]
If $F = \{f\}$ is a singleton fiber, then $c_f = 0$. If $F = \{f, g\}$ has two elements, then $c_g = -c_f$. Thus vanishing at one vertex propagates along every edge of the fiber graph. Since every connected component contains a vertex belonging to a singleton fiber, all coefficients vanish.
\end{proof}

We give a familiar example to illustrate the above construction.

\begin{example} \label{graph for oa}
Let $\C$ be the simplex category on nonempty finite ordered sets $[n] = \{1, \ldots, n\}$, $n \geqslant 1$. Then $\C^+$ consists of injective order-preserving maps (face maps), and $\C^-$ consists of surjective order-preserving maps (degeneracy maps). The unfactorizable morphisms in $\C^-$ are precisely the elementary degeneracy maps $s_i: [n] \to [n-1]$, where $s_i$ identifies the consecutive elements $i$ and $i+1$, for $1\leqslant i < n$. For each elementary degeneracy, every nonempty fiber contains at most two injections, since the degeneracy identifies only two consecutive elements.

If $m > n+1$, then the fiber graph $\Gamma_{n,m}$ is connected. Indeed, an order-preserving injection $f:[n] \hookrightarrow [m]$ is determined by its image, or equivalently by a word with $n$ occupied positions and $m-n$ holes. Two such injections belong to a common two-element fiber precisely when their images differ by exchanging two consecutive positions, one occupied and the other a hole. Every such word can be transformed into any other by adjacent exchanges, so the graph is connected. Moreover, one can show that $\Gamma_{n,m}$ contains a vertex belonging to a singleton fiber. Indeed, let $f = \iota$ be the standard inclusion $\iota(j) = j$ and consider the elementary degeneracy map $s_{n+1}: [m] \to [m-1]$. This map identifies $n+1$ and $n+2$. It is clear that $s_{n+1} \circ f$ remains injective. Moreover, there is no other $g: [n] \hookrightarrow [m]$ such that $s_{n+1} \circ f = s_{n+1} \circ g$. Indeed, $s_{n+1}$ is injective over the image of $s_{n+1} \circ f$, whose unique preimages are $1,\ldots,n$. Thus $\{f\}$ is a singleton fiber. Since each injection belongs to a fiber and $\Gamma_{n,m}$ is connected, Lemma \ref{binary fiber criterion} tells us that $M_{n,m}^{\mathsf T}$ is injective over any commutative ring.

If $m = n +1 $, then each degeneracy map $s_i: [n+1] \to [n]$ determines a unique fiber of size $2$, consisting of the two face maps $d_i$ and $d_{i+1}$, where $d_j$ denotes the increasing injection omitting $j$. Thus the graph $\Gamma_{n, n+1}$ is the path graph of type $A_{n+1}$. Consequently $M_{n,n+1}$ is the vertex-fiber incidence matrix of $A_{n+1}$, so we get the following description:
\[
M_{n,n+1} = \bigl(\delta_{j,i}+\delta_{j,i+1}\bigr)_{\substack{1\leqslant j\leqslant n+1\\1\leqslant i\leqslant n}}
\quad
T_{n,n+1} =
\begin{pmatrix}
1 & 1 & 0 & \cdots & 0 & 0\\
1 & 2 & 1 & \cdots & 0 & 0\\
0 & 1 & 2 & \cdots & 0 & 0\\
\vdots & \vdots & \vdots & \ddots & 1 & 0\\
0 & 0 & 0 & \cdots & 2 & 1\\
0 & 0 & 0 & 0 & 1 & 1
\end{pmatrix}.
\]
Equivalently, $M_{n,n+1}^{\mathsf T}(c_1, \ldots, c_{n+1})^{\mathsf T} = 0$ if and only if $c_i + c_{i+1} = 0$ for $1 \leqslant i \leqslant n$. Consequently, its kernel is free of rank $1$ over any commutative ring, generated by the alternating vector $(1,-1,\ldots,(-1)^n)^{\mathsf T}$.
\end{example}

\section{Descriptions of Hom-spaces between standard modules}

Let $\C$ be a generalized Reedy category and $k$ a commutative ring. For every $x \in \Ob(\C)$, let $P_x = k\C(x, -)$ be the representable functor, $\mathfrak{I}_x$ the submodule generated by morphisms $f: x \to z$ in $\C^-$ such that $d(z) < d(x)$, and $\Delta_x = P_x / \mathfrak{I}_x$. More explicitly, for each object $z \in \Ob(\C)$, the submodule $\mathfrak{I}_x(z) \subseteq k\C(x,z)$ is given by
\[
I_x(z) = \sum_{\substack{y \in \Ob(\C)\\d(y)<d(x)}} \ \sum_{h \in \C^-(x,y)} \operatorname{Im} \left( k\C(y,z) \xrightarrow{- \circ h} k\C(x,z) \right).
\]
Equivalently, $\mathfrak{I}_x(z)$ is the $k$-submodule spanned by all morphisms $f: x \to z$ whose Reedy factorization has a noninvertible negative part. We call $\Delta_x$ the \emph{standard module} associated to the object $x$. By the generalized Reedy structure, $\Delta_x$ is a free $k$-module with a natural basis parameterized by all morphisms in $\C^+(x, -)$ and $\mathfrak{I}_x$ has a basis consisting of all other morphisms in $\C(x, -)$. Thus as $k$-modules, we have $P_x \cong \mathfrak{I}_x \oplus \Delta_x$ and $\Delta_x(z) \cong k\C^+(x, z)$ for all $x, z \in \Ob(\C)$.

Standard modules are intermediate objects between projective modules and irreducible modules. The main goal of this section is to give concrete and transparent descriptions of homomorphisms between standard modules, relying on the combinatorial and linear algebraic constructions described in the previous section. By the following result, apart from the diagonal case $x = y$, we only need to consider $\Hom_{k\C} (\Delta_y, \Delta_x)$ with $d(y) > d(x)$.

\begin{proposition} [Theorem 3.6 \cite{DLL}] \label{semi-orthogonal}
Let $\C$ be a generalized Reedy category, $k$ a commutative ring, and $x, y$ objects in $\C$. Then
\[
\Hom_{k\C} (\Delta_y, \Delta_x) \neq 0 \, \Longrightarrow \, x = y \text{ or } d(y) > d(x).
\]
Moreover, $\End_{k\C} (\Delta_x) \cong kG_x$.
\end{proposition}

We introduce the following notion.

\begin{definition}
A \emph{compatible family of coefficients} on $\C^+(x,y)$ is a map
\[
\epsilon : \C^+(x,y) \longrightarrow k
\]
such that for each representative unfactorizable morphism $h: y \to z$ in $\mathbb{U}_y$ and each connected component $C$ of $\Gamma_{x,y}^h$, one has
\[
\sum_{f \in C} \epsilon(f) = 0.
\]
\end{definition}

\begin{remark}
Since the collection of fibers $\mathcal{F}_{x, y}$ is independent of the choice of representatives, the above definition of compatible families of coefficients is also independent of this choice. Furthermore, besides the above graph-theoretic definition, there are several equivalent ways to describe a compatible family of coefficients on $\C^+(x,y)$:

\begin{enumerate}
\item \textbf{Fiber-sum characterization:} Since fibers $F \in \mathcal{F}_{x, y}$ are precisely connected components in $\Gamma_{x, y}^h$, the function $\epsilon$ is compatible if and only if
\[
\sum_{f \in F} \epsilon(f) = 0 \quad \forall \, F \in \mathcal{F}_{x, y}.
\]

\item \textbf{Orthogonal subspace characterization:} Identifying $\epsilon$ with a linear functional in
\[
(k\C^+(x,y))^{\ast} = \mathrm{Hom}_k(k\C^+(x,y), \,  k),
\]
and let $R_{x, y} \subseteq k\C^+(x,y)$ denote the $k$-submodule generated by
\[
\{ \, \sum_{f \in F} f \mid F \in \mathcal{F}_{x, y} \, \}.
\]
Thus $\epsilon$ is a compatible family of coefficients if and only if it is contained in
\[
R_{x, y}^\perp = \{ \epsilon \in (k\C^+(x,y))^* \mid \epsilon(r) = 0, \, \forall \, r \in R_{x, y} \}.
\]
\end{enumerate}

Note that if a morphism $f \in \C^+(x, y)$ is not a vertex in $\Gamma_{x, y}$, then $f$ does not belong to any fiber in $\mathcal{F}_{x, y}$. Consequently, no compatibility condition involves $\epsilon(f)$, and this coefficient may be chosen arbitrarily.
\end{remark}

The following result provides combinatorial and linear algebraic realizations of homomorphism spaces between standard modules.

\begin{theorem} \label{hom spaces general}
Let $\C$ be a generalized Reedy category and $k$ a commutative ring. Then for all $x, y \in \Ob(\C)$ such that $d(y) > d(x)$, there are natural isomorphisms between the following $k$-modules:
\begin{enumerate}
\item the homomorphism space $\Hom_{k\C}(\Delta_y, \Delta_x)$;

\item the $k$-module consisting of all compatible families of coefficients on $\C^+(x,y)$;

\item the kernel space of the operator
\[
M_{x, y}^{\mathsf T}: (k\C^+(x, y))^{\ast} \longrightarrow (k\mathcal{F}_{x, y})^{\ast}, \quad (M_{x, y}^{\mathsf T} \rho)(F) = \sum_{f \in F} \rho(f).
\]

\item the orthogonal subspace $R_{x, y}^{\perp}$.
\end{enumerate}
\end{theorem}

\begin{proof}
\noindent
\textbf{(1) is isomorphic to (2).} Recall that $\Delta_x = P_x / \mathfrak{I}_x$ and $\Delta_y = P_y / \mathfrak{I}_y$, where $P_x, P_y$ are representable projective $\C$-modules and $\mathfrak{I}_x, \mathfrak{I}_y$ are generated by morphisms in $\C^-$ strictly decreasing degree. Then
\[
\Hom_{k\C}(\Delta_y, \Delta_x) \;\cong\; \{ \phi \in \Hom_{k\C}(P_y, \Delta_x) \mid \mathfrak{I}_y \subseteq \ker \phi \}.
\]
By Yoneda's Lemma,
\[
\Hom_{k\C}(P_y, \Delta_x) \;\cong\; \Delta_x(y) \;\cong\; k\C^+(x, y).
\]
Under this identification, the condition $\mathfrak{I}_y \subseteq \ker \phi$ becomes
\[
\Hom_{k\C}(\Delta_y, \Delta_x) \;\cong\; \Big\{ v = \sum_{f \in \C^+(x, y)} c_f f \;\Big|\; \mathfrak{I}_y \cdot v \subseteq \mathfrak{I}_x \Big\}, \quad (\dag)
\]
where the action $\cdot$ is induced by composition. By Lemma \ref{unfactorizable}, every non-invertible morphism in $\C^-(y,-)$ can be written as a composite $h = h_r \circ \cdots \circ h_1$ of unfactorizable morphisms. Moreover, one can choose $h_1$ to be a representative unfactorizable morphism in $\mathbb{U}_y$. Since $\mathfrak I_x$ is a $\C$-submodule of $P_x$, it follows that $h_1 \cdot v \in \mathfrak I_x$ implies $h \cdot v \in \mathfrak I_x$. Hence the condition $(\dag)$ can be simplified as
\[
\Hom_{k\C}(\Delta_y, \Delta_x) \;\cong\; \Big\{ v = \sum_{f \in \C^+(x, y)} c_f f \;\Big|\; h \cdot v \in \mathfrak{I}_x, \, \forall \, h \in \mathbb{U}_y \Big\}. \quad (\ddag)
\]

Fix $h: y \to z$ in $\mathbb{U}_y$. For $\alpha \in \C(x, z)$, consider the set
\[
h^{\ast}(\alpha) = \{ f \in \C^+(x, y) \mid h \circ f = \alpha \}.
\]
If $\alpha$ is not in $\C^+(x, z)$, then $\alpha$ represents $0$ in $\Delta_x(z) \cong k\C^+(x,z)$, so such terms do not contribute to the condition $(\ddag)$. Thus we only need to consider $\alpha \in \C^+(x, z)$. In this case, $h^{\ast}(\alpha)$ is the fiber $F_{h, \alpha} \in \mathcal{F}_{x, y}$, and the coefficient of $\alpha$ in $h \cdot v$ is
\[
\sum_{f \in F_{h, \alpha}} c_f,
\]
which must be zero since $h \cdot v \in \mathfrak{I}_x$. Running over all $\alpha$ and all representative unfactorizable $h \in \mathbb{U}_y$, we obtain
\[
\sum_{f \in F} c_f = 0
\]
for every fiber $F \in \mathcal{F}_{x, y}$. Hence $(c_f)$ is a compatible family of coefficients.

\medskip
Conversely, given a compatible family $\epsilon: \C^+(x, y) \longrightarrow k$, define
\[
v = \sum_{f \in \C^+(x, y)} \epsilon(f) f \in k\C^+(x, y).
\]
Let $h: y \to z$ be a representative unfactorizable morphism in $\mathbb{U}_y$. Then
\[
h \cdot v = \sum_{f \in \C^+(x, y)} \epsilon(f)\,(h \circ f).
\]
For any morphism $\alpha \in \C^+(x, z)$ appearing in the above expression of $h \cdot v$, its coefficient in $h \cdot v$ is
\[
\sum_{f \in F_{h, \alpha}} \epsilon(f).
\]
By compatibility, this sum is zero. Therefore all coefficients of basis elements in $\C^+(x,z)$ vanish, and hence $h \cdot v \in \mathfrak{I}_x$. Thus $v$ defines an element of $\Hom_{k\C}(\Delta_y, \Delta_x)$.

In summary, the assignments
\[
v = \sum_{f} c_f f \;\longmapsto\; \epsilon(f) = c_f, \qquad \epsilon \;\longmapsto\; \sum_{f} \epsilon(f) f
\]
define mutually inverse $k$-module homomorphisms between $\Hom_{k\C}(\Delta_y, \Delta_x)$ and the set of compatible families of coefficients on $\C^+(x, y)$. Hence we obtain the desired isomorphism.

\medskip
\noindent
\textbf{(2) is isomorphic to (3).} For $\epsilon \in (k\C^+(x,y))^*$ and $F\in\mathcal F_{x,y}$,
\[
(M_{x,y}^{\mathsf T} \epsilon)(F) = \sum_{f \in F}\epsilon(f).
\]
Hence
\[
\ker(M_{x,y}^{\mathsf T}) = \Bigl\{ \epsilon\in (k\C^+(x,y))^* \;\Big|\; \sum_{f\in F}\epsilon(f) = 0, \, \forall F \in \mathcal{F}_{x,y} \Bigr\},
\]
which is exactly the space of compatible families.

\medskip
\noindent
\textbf{(2) is isomorphic to (4).} This has been explained in the previous remark.
\end{proof}

\begin{remark}
The space of compatible families
\[
\Bigl\{
\epsilon \in (k\C^+(y,x))^*
\;\Big|\;
\epsilon \text{ is compatible}
\Bigr\}
\]
is stable under the natural right action of $G_y$ on $\C^+(y,x)$. Indeed, for every representative unfactorizable morphism $h$ and every $\alpha$, composition with $\sigma$ permutes the fibers $F_{h, \alpha}$. Therefore the relations
\[
\sum_{f\in F}\epsilon(f)=0
\]
are preserved under the action of $G_y$. This symmetry can often be used to reduce the number of independent compatibility conditions that need to be verified.
\end{remark}

We describe two classical examples in homology theory.

\begin{example}
Let $\C$ be the simplex category. By Theorem \ref{hom spaces general}, $\Hom_{k\C}(\Delta_m, \Delta_n)$ is identified with $\ker (M_{n,m}^{\mathsf T})$. From Example \ref{graph for oa}:
\begin{enumerate}
\item If $m = n+1$, then $\Hom_{k\C}(\Delta_{n+1}, \Delta_n) \cong k$, generated by the alternating sum over $\Gamma_{n,n+1}$.

\item If $m>n+1$, then $M_{n,m}^{\mathsf T}$ is injective, hence $\Hom_{k\C}(\Delta_m, \Delta_n) = 0$.
\end{enumerate}
Thus one recovers the Dold--Kan vanishing pattern in this setting.
\end{example}

\begin{example}
Let $\C = \Lambda$ be the cyclic category with objects $[n]=\{1,\ldots,n\}$, so that $\Aut_{\C}([n])\cong\mathbb Z/n\mathbb Z$. We give an explicit description of the fiber graph $\Gamma_{n,n+1}$. Let $q = n+1$ and identify $[q]$ with $\mathbb Z/q\mathbb Z$ in its cyclic order. Every morphism in $\C^+(n,q)$ is uniquely determined by a pair $(a,b) \in \mathbb{Z}/q\mathbb{Z} \times \{1,\ldots,n\}$. Indeed, let $f_{a,b}: [n] \hookrightarrow [q]$ be the cyclic order-preserving injection obtained from the cyclic sequence $a, a+1, \ldots, a+n$ by omitting $a+b$, while preserving the displayed order of the remaining elements. Thus $f_{a,b}(1) = a$.

Every unfactorizable cyclic degeneracy identifies two consecutive vertices of the $q$-cycle. Since the image of $f_{a,b}$ omits exactly one vertex, its composite with such a degeneracy is injective precisely when the
identified pair contains the omitted vertex. In this case the corresponding fiber consists of exactly two injections, obtained from one another by moving the occupied endpoint into the hole.

Consequently, for $n\geqslant 2$, the edges of $\Gamma_{n,n+1}$ are precisely $(a,b) \sim (a,b+1)$, $1 \leqslant b <n$, and $(a,1) \sim(a+1,n)$ for $a \in \mathbb Z/q\mathbb Z$. It follows that $\Gamma_{n,n+1}$ is
a cycle with $n(n+1)$ vertices. In particular, it is connected and bipartite, since $n(n+1)$ is even. For $n=1$, the graph consists of two vertices joined by two parallel edges, corresponding to the two indexed cyclic degeneracies; its underlying simple graph is a single edge. By Theorem \ref{hom spaces general}, the kernel of $M_{n,n+1}^{\mathsf T}$ consists of coefficient families $\rho$ satisfying $\rho(f) + \rho(g) = 0$ for every pair of vertices belonging to a common fiber.

Define $\rho(f_{a,b}) = (-1)^{na+b}$. This is well-defined for $a \in \mathbb Z/q\mathbb Z$, since $nq = n(n+1)$ is even. Since $\rho(f_{a,b+1}) = -\rho(f_{a,b})$ and $\rho(f_{a+1,n}) = -\rho(f_{a,1})$, the signs alternate along every edge. It follows that $\rho$ is contained in $\ker(M_{n,n+1}^{\mathsf T})$. Conversely, because the fiber graph is connected, the coefficient at any one vertex determines all other coefficients through the relations $\rho(f) + \rho(g) = 0$. The preceding alternating-sign family shows that these relations are consistent. Therefore, for every commutative ring $k$, we have $\ker(M_{n,n+1}^{\mathsf T}) \cong k$, which is generated by the alternating family $\rho(f_{a,b}) = (-1)^{na+b}$. This gives $\Hom_{k\Lambda} (\Delta_{n+1}, \Delta_n) \cong k$.
\end{example}

For more examples including the category of finite sets, please refer to \cite{Li}.

\section{Hom-orthogonality and orthogonal decomposition}

As before, let $\C$ be a skeletal generalized Reedy category, and $k$ a commutative ring.

\subsection{A criterion of orthogonal decomposition}

Note that the family $\{\Delta_x \mid x \in \Ob(\C)\}$ is semi-orthogonal in the sense that $\Hom_{k\C}(\Delta_y, \Delta_x) = 0$ whenever $d(y) \not\geqslant d(x)$. We say it is \emph{hom-orthogonal} if $\Hom_{k\C}(\Delta_y, \Delta_x) = 0$ for all $x \neq y$. As an immediate consequence of Theorem \ref{hom spaces general}, we obtain:

\begin{corollary}\label{equivalent conditions}
Let $k$ be a commutative ring and $\C$ a skeletal generalized Reedy category. For $x \neq y$, consider the following conditions:
\begin{enumerate}
\item $\Hom_{k\C} (\Delta_y, \Delta_x) = 0$;

\item $M_{x,y}^{\mathsf T}$ is injective;

\item $k\C^+(x,y)$ is spanned by fiber sums $\sum_{f \in F} f$, where $F \in \mathcal{F}_{x,y}$.
\end{enumerate}
Conditions {\rm(1)} and {\rm(2)} are equivalent, and {\rm(3)} implies both of them. If $k$ is a field, all three conditions are equivalent.
\end{corollary}

For some special examples, we obtain a quite simple criterion for the hom-orthogonality of standard modules.

\begin{corollary}\label{eigenvalue criterion}
Let $k$ be a field of characteristic zero. Fix objects $x, y \in \Ob(\C)$ with $d(y) > d(x)$, and assume that all diagonal entries of $T_{x,y}$ equal a constant $r$. Then $\Hom_{k\C}(\Delta_y, \Delta_x) = 0$ if and only if $-r$ is not an eigenvalue of the weighted adjacency matrix $\tilde{A}_{x,y}$.
\end{corollary}

\begin{proof}
By Theorem~\ref{hom spaces general}, $\Hom_{k\C}(\Delta_y, \Delta_x) = 0$ if and only if $M_{x,y}^{\mathsf T}$ is injective. Since $M_{x,y}$ has integer entries and $k$ has characteristic zero, the Cauchy--Binet formula implies that $M_{x,y}^{\mathsf T}$ is injective if and only if $T_{x,y}=M_{x,y}M_{x,y}^{\mathsf T}$ is invertible. Indeed, when $M_{x,y}$ has full row rank, $\det(T_{x,y})=\sum_I\det(M_I)^2$ is a positive integer, where $M_I$ ranges over all maximal square submatrices of $M_{x,y}$. Over a field, this is true if and only if $0$ is not an eigenvalue of $T_{x,y}$. But $T_{x,y} = rI + \tilde{A}_{x,y}$, this happens if and only if $-r$ is not an eigenvalue of $\tilde{A}_{x,y}$.
\end{proof}

It has been shown that for some interesting generalized Reedy categories $\C$, the representation category $\C \Mod$ has an orthogonal decomposition; that is, it is equivalent to the Cartesian product of representation categories of local automorphism groups. We describe a sufficient criterion of this decomposition.

\begin{theorem} \label{decomposition}
Let $k$ be a field, and let $\C$ be a skeletal generalized Reedy category. Suppose that $k\C^-(x, y)$ is a projective left $kG_y$-module for all $x, y \in \Ob(\C)$. Then the functor
\[
\bigl(\Hom_{k\C}(\Delta_x,-)\bigr)_{x\in\Ob(\C)}
\]
induces an equivalence of categories
\[
\C \Mod \;\simeq\; \prod_{x \in \Ob(\C)} kG_x \Mod
\]
if and only if the following conditions hold:
\begin{enumerate}
    \item $\Hom_{k\C}(\Delta_y, \Delta_x) = 0$ whenever $x \neq y$;
    \item $\big|\C^+(x, y)\big| = \big|\C^-(y, x)\big|$ for all $x, y \in \Ob(\C)$.
\end{enumerate}
\end{theorem}

\begin{proof}

\medskip
\noindent
\textbf{The if direction.} Assume (1) and (2). We first check that the family $\{\Delta_x\}_{x \in \Ob(\C)}$ forms a set of projective generators of $\C$-Mod. It suffices to show that each standard module is projective. Indeed, each representable $P_x = k\C(x, -)$ has a finite filtration by modules of the form
\[
\Delta_y \otimes_{kG_y} k\C^-(x, y)
\]
by \cite[Remark 3.12]{DLL}. Since $k\C^-(x, y)$ is a projective left $kG_y$-module, this filtration factor is isomorphic to a direct summand of a direct sum of $\Delta_y$. Consequently, if each $\Delta_y$ is a projective $\C$-module, so is each filtration factor. In that case, $P_x$ can be written as a direct sum of these filtration factors, so standard modules do form a set of projective generators.

If $d(y)$ is minimal, then $\Delta_y = P_y$ is clearly projective. Suppose that $d(y)$ is not minimal, and consider the canonical short exact sequence
\[
0 \longrightarrow \mathfrak{I}_y \longrightarrow P_y \longrightarrow \Delta_y \longrightarrow 0. \quad (\dag)
\]
Note that $\mathfrak{I}_y$ is generated by morphisms in $\C^-(y, -)$ strictly decreasing the degree, and has a finite filtration by modules of the form $\Delta_z \otimes_{kG_z} k\C^-(y, z)$ with $d(z) < d(y)$. By induction on the degree, we may assume that all $\Delta_z$ appearing in this filtration, and hence all filtration factors are projective. Consequently, $\mathfrak{I}_y$ is isomorphic to a direct sum of these projective filtration factors:
\[
\mathfrak{I}_y \cong \bigoplus_{d(z) < d(y)} \Delta_z \otimes_{kG_z} k\C^-(y, z).
\]

For any $z$ with $d(z) < d(y)$, consider the long exact sequence
\[
0 \longrightarrow \Hom_{k\C}(\Delta_y, \Delta_z) \longrightarrow \Hom_{k\C}(P_y, \Delta_z) \longrightarrow \Hom_{k\C}(\mathfrak{I}_y, \Delta_z) \longrightarrow \Ext^1_{k\C}(\Delta_y, \Delta_z) \longrightarrow 0.
\]
We have:
\begin{itemize}
    \item The first term vanishes by (1).

    \item The second term has dimension $\dim_k\Delta_z(y) = |\C^+(z, y)|$.

    \item In the third term, by (1), only the direct summand $\Delta_z \otimes_{kG_z} k\C^-(y, z)$ of $\mathfrak{I}_y$ contributes. Thus we have
\begin{align*}
\dim_k \Hom_{k\C} (\mathfrak{I}_y, \Delta_z) & = \dim_k \Hom_{k\C} (\Delta_z \otimes_{kG_z} k\C^-(y, z), \, \Delta_z)\\
 & = \dim_k \Hom_{kG_z} (k\C^-(y, z), \, \Hom_{k\C} (\Delta_z, \Delta_z))\\
 & = \dim_k \Hom_{kG_z} (k\C^-(y, z), kG_z)\\
 & = \dim_k k\C^-(y, z)
\end{align*}
which is $|\C^-(y, z)| = |\C^+(z, y)|$.
\end{itemize}
Consequently, the map
\[
\Hom_{k\C}(P_y, \Delta_z) \longrightarrow \Hom_{k\C}(\mathfrak{I}_y, \Delta_z)
\]
is surjective, and hence $\Ext^1_{k\C}(\Delta_y, \Delta_z) = 0$, which forces
\[
\Ext^1_{k\C}(\Delta_y, \, \Delta_z \otimes_{kG_z} k\C^-(y, z)) = 0
\]
for any $z$ with $d(z) < d(y)$. Since $\mathfrak{I}_y$ is isomorphic to a direct sum of these filtration factors, it follows by induction that
\[
\Ext^1_{k\C}(\Delta_y, \mathfrak{I}_y) = 0.
\]
Therefore the short exact sequence $(\dag)$ splits, and $\Delta_y$ is projective.

Thus as claimed, the family $\{\Delta_x\}_{x \in \Ob(\C)}$ forms a set of projective generators of $\C$-Mod. Moreover, one has $\End_{k\C}(\Delta_x) \cong kG_x$ and $\Hom_{k\C}(\Delta_y, \Delta_x) = 0$ when $x \neq y$, so
\[
\End_{k\C}\Big(\bigoplus_{x \in \Ob(\C)} \Delta_x\Big) \cong \prod_{x \in \Ob(\C)} kG_x.
\]
The desired equivalence then follows from Morita theory for a family of finitely generated projective generators.

\medskip
\noindent
\textbf{The converse direction.}
The first condition is clearly necessary: if the functor induces an equivalence, then the images of distinct summands must be orthogonal, hence
\[
\Hom_{k\C}(\Delta_y, \Delta_x) = 0
\]
when $x \neq y$. Moreover, each standard module $\Delta_x$ is projective. Indeed, under the assumed equivalence
\[
\C \Mod \simeq \prod_{x \in \Ob(\C)} kG_x \Mod,
\]
each $\Delta_x$ corresponds to a projective generator of the $x$-component, hence is projective in $\C$-Mod. In particular, for all $x, y \in \Ob(\C)$, one has
\[
\Ext^1_{k\C}(\Delta_y, \Delta_x) = 0.
\]

Now we show Condition (2). This is obviously true if $x = y$. Suppose $x \neq y$, and consider the short exact sequence
\[
0 \longrightarrow \Hom_{k\C}(\Delta_y, \Delta_x)
\longrightarrow \Hom_{k\C}(P_y, \Delta_x)
\longrightarrow \Hom_{k\C}(\mathfrak{I}_y, \Delta_x)
\longrightarrow \Ext^1_{k\C}(\Delta_y, \Delta_x)
\longrightarrow 0.
\]
Since the first and last terms vanish, the middle terms have equal dimension. By the same identification as before,
\[
\dim_k \Hom_{k\C}(P_y, \Delta_x) = |\C^+(x,y)| \quad \text{and} \quad \dim_k \Hom_{k\C}(\mathfrak{I}_y, \Delta_x) = |\C^-(y,x)|.
\]
Therefore, $|\C^+(x,y)| = |\C^-(y,x)|$.
\end{proof}

In the rest of this section we describe several applications of Theorem \ref{decomposition}.

\subsection{Span categories} \label{weighted span}

Let $\D$ be a skeletal EI category. We define a relation $\preccurlyeq$ on $\Ob(\D)$ in the following way: given $x, y \in \Ob(\D)$, set $x \preccurlyeq y$ if $\D(x, y) \neq \varnothing$. It is easy to verify that $\preccurlyeq$ is a partial order. Throughout this subsection suppose that $\D$ satisfies the following conditions:
\begin{enumerate}
\item[(a)] the poset $(\Ob(\D), \preccurlyeq)$ is countable and artinian, every principal lower ideal is finite, and there are only finitely many objects of each finite height;

\item[(b)] $\D(x, y)$ is a finite set for all $x, y \in \Ob(\D)$;

\item[(c)] every pair of morphisms with common codomain admits a pullback;

\item[(d)] every morphism in $\D$ is monic.
\end{enumerate}

One can define a degree map $d: \Ob(\D) \to \mathbb{N}$ as follows. For any minimal object $x$ in $(\Ob(\D), \preccurlyeq)$ (which always exists since $\preccurlyeq$ is artinian), define $d(x) = 1$. Removing these minimal objects, we get a full subcategory whose underlying object set is again an artinian poset, so we can define the degree of minimal objects in this subcategory to be 2. Recursively we obtain the desired degree function. Since each principal lower ideal is finite, every object receives a finite degree, and condition (a) ensures that there are only finitely many objects in each degree. Moreover, conditions (a) and (b) imply that each morphism set in the span category defined below is finite.

We define a category $\C$ as follows. It has the same objects as $\D$. For $x, y \in \Ob(\C)$, a morphism $x \to y$ is an equivalence class of spans
\[
x \xleftarrow{f} z \xrightarrow{g} y,
\]
where $f: z \to x$ and $g: z \to y$ are morphisms in $\D$. Two pairs $(f, g)$ and $(f', g')$ are declared equivalent if there exists an automorphism $\varphi: z \to z$ in $\D$ such that the following diagram commutes
\[
\begin{tikzcd}
& z \arrow[dl,"f"'] \arrow[dr,"g"] \arrow[d,"\varphi"] & \\
x & z \arrow[l,"f'"] \arrow[r,"g'"'] & y.
\end{tikzcd}
\]

Given two morphisms in $\C$
\[
x \xleftarrow{f} z \xrightarrow{g} y, \qquad y \xleftarrow{f'} z' \xrightarrow{g'} w,
\]
we form the pullback diagram
\[
\begin{tikzcd}
& z\times_y z' \arrow[dl,"p"'] \arrow[dr,"p'"] & \\
z \arrow[r,"g"'] & y & z' \arrow[l,"f'"]
\end{tikzcd}
\]
and define their composite to be
\[
[f', g'] \circ [f, g] = [f \circ p, \; g' \circ p'].
\]
The resulting construction is independent of the choice of pullback representative, since any two pullbacks of $g$ and $f'$ are uniquely isomorphic in $\D$. Associativity follows from the universal property of pullbacks.

Now we define a wide subcategory $\C^+$ of $\C$ consisting of morphisms of the form $[\id, g]$, and a wide subcategory $\C^-$ of $\C$ consisting of morphisms of the form $[f, \id]$. It is straightforward to check that $\C$ equipped with the degree function induced from $(\Ob(\D), \preccurlyeq)$, is a generalized Reedy category. Indeed, every span admits a factorization
\[
[f, g] = [\id, g] \circ [f, \id],
\]
unique up to isomorphism of the intermediate object, and non-isomorphisms in $\C^+$ (resp. $\C^-$) strictly raise (resp. lower) degree. We omit the routine verification of the remaining axioms.

\begin{lemma} \label{elementary properties}
Let $\C$ and $\D$ be as above, and fix distinct $x,y \in \Ob(\C)$ with $d(y) > d(x)$. Then:
\begin{enumerate}
\item Each fiber $F \in \mathcal{F}(x,y)$ is a singleton.

\item A morphism represented by $y \xleftarrow{f} x \xrightarrow{\id} x$ has a unique right inverse in $\C$ represented by $x \xleftarrow{\id} x \xrightarrow{f} y$.

\item Every morphism in $\C^+(x,y)$ is contained in a fiber $F \in \mathcal{F}_{x,y}$.
\end{enumerate}
\end{lemma}

\begin{proof}

\noindent
(1) Let $\hat{f}, \hat{g} \in \C^+(x,y)$ lie in the same fiber. Then there exists an unfactorizable $\hat{h}\in \C^-(y,z)$ such that
\[
\hat{h} \circ \hat{f} = \hat{h} \circ \hat{g} \in \C^+(x,z).
\]
Choose representatives
\[
\hat{f} = [x \xleftarrow{\id} x \xrightarrow{f} y],\quad \hat{g} = [x \xleftarrow{\id} x \xrightarrow{g} y], \quad \hat{h} = [y \xleftarrow{h} z \xrightarrow{\id} z].
\]
Computing $\hat{h} \circ \hat{f}$ via pullback, and using that all morphisms in $\D$ are monic, we may choose a representative with apex $x$:
\[
\xymatrix{
& x \ar[dl]_{\id} \ar[dr]^{\delta_f} \\
x \ar[r]_f & y & z \ar[l]^h
}
\quad \text{so that} \quad f=h\circ \delta_f.
\]
Similarly, $g = h \circ \delta_g$. Since $\hat{h} \circ \hat{f} = \hat{h} \circ \hat{g}$ in $\C^+(x,z)$, it follows that $(\id_x, \delta_f)$ and $(\id_x, \delta_g)$ are equivalent, forcing $\delta_f = \delta_g$. Hence $f = g$, and therefore $\hat{f} = \hat{g}$. This proves that each fiber is a singleton.

\medskip

\noindent
(2) Consider the span $y \xleftarrow{f} x \xrightarrow{\id} x$. Since $f$ is monic, the pullback of $x \xrightarrow{f} y \xleftarrow{f} x$ has apex canonically isomorphic to $x$, so
\[
x \xleftarrow{\id} x \xrightarrow{f} y
\]
defines a right inverse. Uniqueness follows from the universal property of pullbacks.

\medskip

\noindent
(3) Let $\hat{f} \in \C^+(x,y)$ and consider $\hat{f}^\circ=[y \xleftarrow{f} x \xrightarrow{\id} x] \in \C^-(y,x)$. By (2), one has $\hat{f}^\circ \circ \hat{f} = \id_x$. Factor $\hat{f}^\circ = \hat{g}^ \circ \circ \hat{h}^\circ$ in $\C^-$ with $\hat{h}^\circ$ unfactorizable. Then
\[
\hat{g}^\circ \circ (\hat{h}^\circ \circ \hat{f}) = \id_x.
\]
By uniqueness of right inverses, $\hat{h}^\circ \circ \hat{f}$ is represented by $x \xleftarrow{\id} x \xrightarrow{g} z$, and hence lies in $\C^+(x, z)$. Thus $\hat{f}$ lies in a fiber in $\mathcal{F}_{x,y}$.
\end{proof}

Now we can describe a strengthened version of \cite[Theorem 5.6]{DLL}.

\begin{theorem}
Let $\C$ be a as above and $k$ a commutative ring. Then
\[
\Hom_{k\C} (\Delta_y, \Delta_x) = 0
\]
whenever $x \neq y$.

In particular, if $k$ is a field, and $k\C^-(x, y)$ is a projective $kG_y$-module for all $x, y \in \Ob(\C)$, then
\[
\C \Mod \, \simeq \, \prod_{x \in \Ob(\C)} kG_x \Mod.
\]
\end{theorem}

\begin{proof}
By Lemma \ref{elementary properties}, every morphism in $\C^+(x,y)$ belongs to a fiber $F \in \mathcal{F}_{x,y}$, and each such fiber is a singleton; hence fiber-sum generators coincide with the standard basis of $k\C^+(x,y)$, so condition (3) of Corollary \ref{equivalent conditions} holds and we obtain the first conclusion.

For the second statement, we observe that $|\C^+(x, y)| = |\D(x, y)| = |\C^-(y, x)|$, so Theorem \ref{decomposition} applies and yields the stated categorical equivalence.
\end{proof}

\subsection{Mackey functors} \label{Mackey functors}

In this subsection we use the philosophy showing Theorem \ref{decomposition} to give a new proof for the semisimplicity of Mackey functors established in \cite{TW}.

Let $G$ be a finite group, and let $\mathcal{O}$ be a skeletal orbit category of $G$, whose objects are transitive $G$-sets $G/H$ with $H$ ranging over a set of representatives of conjugacy classes of subgroups. In general $\mathcal{O}$ does not have enough pullbacks, so we cannot use it to define the category of spans. Instead, we define a $k$-linear category: the \emph{Mackey category} $\underline{\C}$. It has the same objects as $\mathcal{O}$. Given two objects $G/H$ and $G/K$, the morphism set $\underline{\C}(G/H, G/K)$ is the free $k$-module spanned by equivalence classes of spans (defined as in the previous subsection)
\[
G/H \xleftarrow{f} G/L \xrightarrow{g} G/K,
\]
where $f$ and $g$ are $G$-equivariant maps. Composition in $\underline{\C}$ is the $k$-linearization of composition of spans. Explicitly, Given spans
\[
\xymatrix{
G/H & G/L \ar[r]^-g \ar[l]_-f & G/K \quad \mathrm{and} \quad G/K & G/A \ar[r]^-{g'} \ar[l]_-{f'} & G/B,
}
\]
form the pullback in the category of finite $G$-sets:
\[
G/L \times_{G/K} G/A \cong \bigsqcup_i G/H_i,
\]
and define the composite to be
\[
(f',g') \circ (f,g) = \sum_i (f p_i, \, g' q_i),
\]
where $p_i, q_i$ are the composites
\[
\xymatrix{
G/H_i \ar[r]^-{\mathrm{inc}} & \bigsqcup_j G/H_j \ar[r]^p & G/L, & G/H_i \ar[r]^-{\mathrm{inc}} & \bigsqcup_j G/H_j \ar[r]^q & G/A.
}
\]
Since $\mathcal{O}$ is skeletal, every transitive component of $G/L \times_{G/K} G/A$ is canonically identified with a unique object $G/H_i$ of $\mathcal{O}$.

\begin{remark}
We observe that the isomorphism classes of spans form a basis of $\underline{\C} (x,y)$, and all compositions are computed on this basis and extended $k$-linearly. The Mackey category admits a generalized $k$-linear Reedy structure in the sense of \cite{DLL} whose positive and negative subcategories are the $k$-linearizations of $\mathcal{O}$ and $\mathcal{O}^{\op}$, respectively. Therefore, we can still define standard modules in a similar way. For details, see \cite{DLL}.
\end{remark}

Now we prove that standard modules are hom-orthogonal.

\begin{lemma} \label{mackey-orthogonality}
Let $k$ be a commutative ring. For $x = G/H$ and $y = G/K$ with $x \neq y$, one has
\[
\Hom_{\underline{\C}}(\Delta_y, \Delta_x) = 0.
\]
\end{lemma}

\begin{proof}
We can canonically identify $\underline{\C}^+(x, y)$ with $k\mathcal{O}(x, y)$ and $\underline{\C}^-(y, x)$ with $k\mathcal{O}^{\op}(y, x) = k\mathcal{O}(x, y)$ as $k$-modules. Let $f_1, \ldots, f_n$ be all morphisms in $\mathcal{O}(x, y)$. Then every morphism $\alpha \in \underline{\C}^+(x, y)$ can be written as a linear combination
\[
\alpha = \sum_{i=1}^n a_i f_i, \quad a_i \in k.
\]
Furthermore, the left ideal $\mathfrak{I}_x$ (resp., $\mathfrak{I}_y$) of $\underline{\C}$ is generated by non-isomorphisms in $\mathcal{O}^{\op} (x, -)$ (resp., $\mathcal{O}^{\op}(y, -)$).

The same argument proving Theorem \ref{hom spaces general} applies verbatim in the $k$-linear setting (since the proof depends only on span composition and $k$-linearity of the basis), so we have
\[
\Hom_{\underline{\C}}(\Delta_y, \Delta_x) \;\cong\; \Big\{ \alpha = \sum_{i = 1}^n a_i f_i \;\Big|\; \mathfrak{I}_y \cdot \alpha \subseteq \mathfrak{I}_x \Big\}.
\]
To show the conclusion, it suffices to find a morphism $g \in \mathcal{O}^{\op} (y, x)$, viewed as an element in $\mathfrak{I}_y$, such that $g \cdot \alpha$ is nonzero in $\underline{\C}(x,-)/\mathfrak{I}_x$ when $\alpha \neq 0$. Without loss of generality, assume that $a_1 \neq 0$. We show that this is true when take $g$ to be the morphism in $\mathcal{O}^{\op}(y, x)$ corresponded to $f_1$.

Identifying $f_i$ and $g$ with spans
\[
f_i : x \xleftarrow{\id} x \xrightarrow{f_i} y, \qquad
g : y \xleftarrow{f_1} x \xrightarrow{\id} x,
\]
the composite $g \cdot f_i$ is represented by the pullback
\[
\xymatrix{
& x \times_y x \ar[dl] \ar[dr] \\
x \ar[r]_{f_i} & y & x \ar[l]^{f_1}.
}
\]
This pullback decomposes as a disjoint union of transitive $G$-sets, each of which contributes a span $x \xleftarrow{} O \xrightarrow{} x$.

When $i=1$, the pullback contains the diagonal component
\[
\Delta=\{(u,u)\in x\times x\},
\]
which is isomorphic to $x$. The corresponded span is
\[
x \xleftarrow{\id} x \xrightarrow{\id} x,
\]
so the identity span occurs in $g\cdot f_1$.

We claim that the identity span does not occur in $g\cdot f_i$ for any $i\neq 1$. Indeed, suppose that some transitive component $O\subseteq x\times_y x$ contributes a span equivalent to the identity span. Let
\[
p: O \to x,\qquad q: O \to x
\]
be the two projections. Then there exists an isomorphism of spans between $x \xleftarrow{p} O \xrightarrow{q} x$ and $x \xleftarrow{\id} x \xrightarrow{\id} x$. In particular, if $\phi: x \to O$ denotes the induced isomorphism between the middle objects, the commutativity with the left legs gives $p \circ \phi = \id_x$. Similarly, the commutativity with the right legs gives $q \circ \phi = \id_x$. Therefore, after identifying $O$ with $x$ via $\phi$, both projections are the identity map.

On the other hand, the pullback condition gives
\[
f_i \circ p = f_1 \circ q.
\]
Composing with $\phi$, we obtain
\[
f_i \circ (p \circ \phi)=f_1 \circ (q \circ \phi),
\]
and hence $f_i = f_1$. Therefore, if $i\neq 1$, no component of the pullback can contribute the identity span.

Consequently, in the composite
\[
g \cdot \alpha = \sum_{i=1}^n a_i (g \cdot f_i),
\]
the coefficient of the identity span is exactly $a_1 \neq 0$. Since the identity span survives in $\underline{\C}(x,-) / \mathfrak I_x$, it follows that $g\cdot \alpha$ is nonzero in this quotient. Therefore $\mathfrak I_y \cdot \alpha \nsubseteq \mathfrak I_x$, and the result follows.
\end{proof}

By definition, a \emph{Mackey functor} is a covariant functor from $\C$ to category of $k$-modules. For more details on these functors, in particular a categorical simplification and generalization, one can refer to \cite{PS, TW}.

The following result was proved in \cite{TW}.

\begin{theorem}
Let $G$ be a finite group, and $k$ be a field such that $|G|$ is invertible in $k$. Then
\[
\underline{\C} \Mod \, \simeq \, \prod_{x \in \Ob(\underline{\C})} kG_x \Mod.
\]
In particular, $\underline{\C} \Mod$ is semisimple.
\end{theorem}

\begin{proof}
Since $|G|$ is invertible in $k$, each group algebra $kG_y$ with $G_y = N_G(H)/H$ is a finite-dimensional semisimple $k$-algebra. Hence every left $kG_y$-module is projective, in particular $k\underline{\C}^-(x,y)$ is a projective left $kG_y$-module. We may therefore apply the same argument as in the proof of Theorem \ref{decomposition} to yield the equivalence.
\end{proof}

\subsection{The category $\VA_q$} \label{VAq}

Our previous applications focus on extracting representation theoretic results via combinatorial or linear algebraic methods. In this subsection we describe an application in the reverse direction. Throughout this subsection let $\C = \VA_q$ be the category of finite dimensional vector spaces and linear maps over a finite field $\mathbb{F}_q$, and let $k$ be a field such that $q$ is invertible in $k$.

Given $n \geqslant m$, let $V_m = \mathbb{F}_q^m$ and $V_n = \mathbb{F}_q^n$. Take
\[
\mathbb{U}_n = \{ g_L : V_n \twoheadrightarrow V_n/L \mid L \subseteq V_n \text{ a line} \}
\]
be a complete set of representative unfactorizable morphisms in $\C^-(V_n, V_{n-1})$ modulo the free left action of $G_{n-1} = \mathrm{GL}_{n-1} (\mathbb{F}_q)$. We have the following elementary observations about the fiber graph $\Gamma_{m, n}$, whose proofs are straightforward.
\begin{enumerate}
\item Two distinct morphisms $f,g \in \C^+(V_m, V_n)$ are adjacent if and only if there exists a line $L \subseteq V_n$ such that $f - g \in L \otimes (V_m)^*$ and $L \cap \Ima(f) = 0$.

\item For each $g_L \in \mathbb{U}_n$ and $f \in \C^+(V_m, V_n)$ with $L \cap \Ima(f) = 0$, the fiber containing $f$ is
\[
F_L(f) = \{\, f + \ell \otimes \phi \mid \phi \in (V_m)^* \,\}
\]
with cardinality $q^m$.

\item For $f, g \in \C^+(V_m, V_n)$, the number of fibers containing both $f$ and $g$ is
\[
N(f,g) =
\begin{cases}
r_{m, n} = \frac{q^n - q^m}{q - 1}, & f = g,\\
1, & \rank(f - g) = 1 \text{ and } \Ima(f-g) \cap \Ima(f) = 0,\\
0, & \text{otherwise}.
\end{cases}
\]
\end{enumerate}

Consequently, the graph $\Gamma_{m,n}$ is regular of degree
\[
r_{m, n} (q^m - 1) = \frac{q^n - q^m}{q - 1}\,(q^m - 1).
\]
Indeed, since each of the $r_{m, n}$ fibers containing $f$ has cardinality $q^m$, and distinct fibers intersect only in $f$, the number of neighbors of $f$ is $r_{m, n}(q^m - 1)$. Moreover, every morphism in $\C^+(V_m, V_n)$ is a vertex in $\Gamma_{m, n}$, so the weighted adjacency matrix $\tilde{A}_{m, n}$ of $\C^+(V_m, V_n)$ coincides with the reduced weight adjacency matrix $A_{m, n}$.

Kuhn proved the following striking result:

\begin{theorem} [Theorem 1.1 \cite{Kuhn}]
Let $k$ be a commutative ring such that $q$ is invertible in $k$. Then
\[
\C \Mod \;\simeq\; \prod_{n \geqslant 0} k\mathrm{GL}_n(\mathbb{F}_q) \Mod.
\]
\end{theorem}

Denote $\Gr(m,n)$ the Grassmanian and let
\[
\pi: \C^+(V_m, V_n) \to \Gr(m,n), \quad f \mapsto \Ima(f)
\]
be the natural map. Combining his theorem and our results, we obtain the following corollary.

\begin{corollary}\label{grassmann spectral}
Let $k$ be a field of characteristic zero, and let $0\leqslant m<n$. Then:
\begin{enumerate}
\item The matrix $T_{m,n}$ is invertible.

\item The spectrum of $A_{m, n}$ does not contain $-r_{m, n}$.

\item One has a decomposition of $A_{m, n}$-invariant spaces
\[
(k\C^+(V_m, V_n))^{\ast} = V_{\mathrm{base}} \oplus V_{\mathrm{fib}},
\]
where
\[
V_{\mathrm{base}} = \{ \epsilon \mid \epsilon(f) = \epsilon(g) \text{ if } \Ima(f) = \Ima(g) \},
\]
\[
V_{\mathrm{fib}} = \{ \epsilon \mid \sum_{f \in \pi^{-1}(U)} \epsilon(f) = 0, \, \forall \, U \in \Gr(m, n) \}.
\]
Consequently,
\[
\mathrm{Spec}(A_{m, n}) = \mathrm{Spec}(A_{\Gr}) \;\cup\; \mathrm{Spec}(A|_{V_{\mathrm{fib}}}).
\]
\end{enumerate}
\end{corollary}

\begin{proof}
(1) By \cite[Theorem 2.1 and Proposition 2.6]{Kuhn}, there is a central idempotent $e_m^G$ in $k\End(V_m)$ which annihilates all singular endomorphisms and is congruent to the identity modulo their $k$-span. Thus $P_m e_m^G \cong \Delta_m$, since the latter is the quotient of $P_m$ by the submodule generated by maps of rank less than $m$. Kuhn's Proposition 2.6 shows that these modules are pairwise Hom-orthogonal. Hence $\ker(M_{m,n}^{\mathsf T})=0$ by Theorem \ref{hom spaces general}. Therefore $M_{m,n}^{\mathsf T}$ is injective. Consequently, by the Cauchy--Binet formula in characteristic zero, the Gram matrix $T_{m,n}$ is invertible.

\medskip
\noindent
(2) This follows from Corollary \ref{eigenvalue criterion} since all diagonal entries of $T_{m, n}$ are $r_{m, n}$.

\medskip
\noindent
(3) Every function $\epsilon \in (k\C^+(V_m, V_n))^{\ast}$ admits a unique decomposition
\[
\epsilon = \epsilon_{\mathrm{base}} + \epsilon_{\mathrm{fib}},
\]
where $\epsilon_{\mathrm{base}}$ is constant on each Grassmann fiber $\pi^{-1}(U)$, and $\epsilon_{\mathrm{fib}}$ has zero sum on each Grassmann fiber. Explicitly, by setting $N_m=|\mathrm{GL}_m(\mathbb F_q)|$, which is invertible in $k$, we may take
\[
\epsilon_{\mathrm{base}}(f)=\frac{1}{N_m}\sum_{\Ima(g)=\Ima(f)}\epsilon(g),
\qquad \epsilon_{\mathrm{fib}}=\epsilon-\epsilon_{\mathrm{base}}.
\]
This gives the desired decomposition as vector spaces. The group $\mathrm{GL}_m(\mathbb F_q)$ acts by precomposition on $\C^+(V_m,V_n)$, and the adjacency operator $A_{m,n}$ commutes with this action. Thus it preserves $V_{\mathrm{base}}$, the subspace of invariant functions. Since $A_{m,n}$ is symmetric, it also preserves its orthogonal complement $V_{\mathrm{fib}}$.

On $V_{\mathrm{base}}$, $A_{m, n}$ depends only on intersections of subspaces, so it factors through a linear operator $A_{\mathrm{Gr}}$ on $k^{\Gr(m,n)}$. The desired identity
\[
\mathrm{Spec}(A_{m, n}) = \mathrm{Spec}(A_{\mathrm{Gr}}) \;\cup\; \mathrm{Spec}(A|_{V_{\mathrm{fib}}})
\]
follows from this observation.
\end{proof}

\begin{remark}
Assume that $k$ is a field of characteristic zero. Then the condition
\[
-r_{m, n} \notin \mathrm{Spec}(A_{m, n}), \, \forall \, m < n
\]
is equivalent to Kuhn’s decomposition theorem for $\VA_q$. Indeed, by Corollary \ref{eigenvalue criterion}, one has
\[
-r_{m, n} \notin \mathrm{Spec}(A_{m, n}), \, \forall \, m < n \, \Longleftrightarrow \Hom_{k\C}(\Delta_n, \Delta_m) = 0,  \, \forall \, m < n.
\]
Since $\VA_q$ satisfies the symmetry condition $|\C^+(V_m, V_n)| = |\C^-(V_{n}, V_{m})|$, Theorem \ref{decomposition} then yields
\[
\C \Mod \simeq \prod_{n \geqslant 0} k\mathrm{GL}_n(\mathbb{F}_q) \Mod.
\]
Over an arbitrary field in which $q$ is invertible, the spectral condition still implies the decomposition equivalence, although the converse may fail in positive characteristic.
\end{remark}

\begin{remark}
Under the hypotheses of Corollary \ref{grassmann spectral}, the decomposition
\[
(k\C^+(V_m,V_n))^{*} = V_{\mathrm{base}} \oplus V_{\mathrm{fib}}
\]
shows that part of the spectrum of $A_{m,n}$ is governed entirely by the geometry of the Grassmannian $\mathrm{Gr}(m,n)$. Indeed, the projection
\[
\pi: \C^+(V_m, V_n) \to \mathrm{Gr}(m,n), \qquad f \mapsto \Ima(f),
\]
identifies $V_{\mathrm{base}}$  with the space of functions on $\mathrm{Gr}(m,n)$. Under this identification, the restriction of $A_{m,n}$ to $V_{\mathrm{base}}$ becomes an operator $A_{\mathrm{Gr}}$ depending only on incidences among $m$-dimensional subspaces of $V_n$.

Consequently, $\mathrm{Spec}(A_{\mathrm{Gr}}) \subseteq \mathrm{Spec}(A_{m,n})$, so the study of $A_{m,n}$ naturally splits into a ``Grassmannian part" and a ``fiber part". In particular, known results on Grassmann graphs and the Grassmann association scheme determine a substantial portion of the spectrum of $A_{m,n}$.
\end{remark}

\section{A uniform Dold-Kan correspondence}

In this section we consider representations of the category $\T$ of finite rooted trees whose morphisms are composites of embeddings and admissible contractions. Recall that a \emph{rooted tree} is a poset $(A,\leqslant)$ satisfying:
\begin{enumerate}
\item $A$ has a unique minimal element, called the \emph{root};
\item for every $x \in A$, the \emph{principal ideal}
\[
A_{\leqslant x} = \{y \in A \mid y \leqslant x\}
\]
is a chain.
\end{enumerate}
Let $V(A)$ and $E(A)$ denote the set of vertices and edges of $A$ respectively.

Throughout this section, all rooted trees are regarded as posets satisfying the two conditions mentioned above. All morphisms are understood in the category of posets. We do not use the graph-theoretic notion of trees unless explicitly stated.

\subsection{Reedy structure}

In this subsection, we show that $\T$ is a generalized Reedy category.

\begin{definition}
Let $A$ and $B$ be two rooted trees. A map $f: A \to B$ is called an \emph{embedding} if it is injective, order-preserving, and order-reflecting; that is, for all $x, y \in A$, we have $f(x) \leqslant f(y)$ in $B$ if and only if $x \leqslant y$ in $A$.
\end{definition}

\begin{remark}
In our setup, embeddings are not required to preserve covering relations or adjacency of vertices. This convention allows the present framework to recover classical categories. For example, the full subcategory consisting of finite chains is naturally identified with the simplex category, where embeddings correspond to face maps.
\end{remark}

We may view $A$ as a category in a natural way. Consequently, an order-preserving map is a functor, and it is an embedding if and only if the corresponded functor is fully faithful and injective on objects.

Note that contracting an edge $a \lessdot b$ in a rooted tree again produces a rooted tree. Thus we make the following definition.

\begin{definition}
Let $a \lessdot b$ be an edge in a rooted tree $A$. The \emph{elementary contraction} of $A$ along $a \lessdot b$ is the rooted tree obtained by contracting this edge. We also call the natural map $\pi: A \to A/\{a, b\}$ an \emph{elementary contraction} along the edge $a \lessdot b$. More generally, a morphism $\pi': A \to C$ between two rooted trees is called an \emph{elementary contraction} if there is an isomorphism $\alpha: A/\{a, b\} \to C$ such that $\pi' = \alpha \circ \pi$.
\end{definition}

\begin{lemma} \label{comparison lemma}
Let $A$ be a rooted tree and let $\pi: A \twoheadrightarrow A/\{a,b\}$ be the elementary contraction along an edge $a \lessdot b$. For vertices $x,y \in A$ with $\pi(x)\neq \pi(y)$, one has $\pi(x) \leqslant \pi(y)$ in $A/\{a,b\}$ if and only if one of the following holds:
\begin{enumerate}
\item $x \leqslant y$ in $A$;
\item $x = b$ and $a \leqslant y$ in $A$.
\end{enumerate}
\end{lemma}

\begin{proof}
The map $\pi$ identifies $a$ and $b$ into a single vertex $[a]=[b]$ and fixes every other vertex. Since $\pi$ is order-preserving, if $x \leqslant y$ in $A$, then $\pi(x) \leqslant \pi(y)$ in $A/\{a,b\}$. Moreover, if $x=b$ and $a \leqslant y$, then
\[
\pi(x)=\pi(b)=\pi(a)\leqslant \pi(y),
\]
so condition (2) also implies $\pi(x) \leqslant \pi(y)$.

Conversely, suppose that $\pi(x) \leqslant \pi(y)$ and $\pi(x) \neq \pi(y)$. By construction, the order on $A/\{a,b\}$ is the smallest partial order making $\pi$ order-preserving. Thus there exists a sequence
\[
x = x_0, x_1, \ldots, x_n = y
\]
such that for each $i$, either $x_i \leqslant x_{i+1}$ in $A$, or $\{x_i, x_{i+1}\} = \{a, b\}$. Since $a \lessdot b$, we have $a \leqslant b$, and every occurrence of the pair $(a,b)$ may be replaced by the relation $a \leqslant b$. Consequently, the only way to obtain a comparison in the quotient that does not already hold in $A$ is to replace an occurrence of $a$ by $b$ at the left end of a relation. Hence either $x \leqslant y$ in $A$, or $x = b$ and $a \leqslant y$ in $A$, as required.
\end{proof}

\begin{definition}
An edge $a \lessdot b$ in $A$ is \emph{admissible} if $a$ is the root, or $b$ is the unique immediate successor of $a$. In this case, we say that the corresponded elementary contraction is \emph{admissible}.
\end{definition}

Denote by $E_{\mathrm{ad}}(A) \subseteq E(A)$ the set of admissible edges of $A$.

\begin{lemma} \label{admissibility}
Let $A$ be a rooted tree, $\iota: B \hookrightarrow A$ an embedding, and $x \lessdot y$ a covering relation in $B$. If $\iota(x) < \iota(y)$ is still a covering relation in $A$ and is admissible, then $x \lessdot y$ is admissible in $B$.
\end{lemma}

\begin{proof}
If $\iota(x)$ is the root of $A$, then $x$ is the root of $B$, so $x \lessdot y$ is admissible. Suppose now that $\iota(y)$ is the unique immediate successor of $\iota(x)$ in $A$. Let $z$ be an immediate successor of $x$ in $B$. Since $x < z$, we have $\iota(x) < \iota(z)$. But $\iota(y)$ is the unique child of $\iota(x)$ in $A$, it follows that $\iota(y) \leqslant \iota(z)$. If $z \neq y$, then $z$ and $y$ are incomparable in $B$, so $\iota(z)$ and $\iota(y)$ are incomparable in $A$, contradicting $\iota(y) \leqslant \iota(z)$. Thus we must have $z = y$, namely $y$ is the unique immediate successor of $x$ in $B$. Consequently, the edge $x \lessdot y$ is admissible.
\end{proof}

The following exchange lemma is the cornerstone of the Reedy structure.

\begin{lemma} \label{exchange}
Let $A$ be a rooted tree, $\iota : B \hookrightarrow A$ an embedding, and let
\[
\pi: A \twoheadrightarrow A/\{a,b\}
\]
be the elementary contraction along an admissible edge $a \lessdot b$. Then the composite $\pi \circ \iota$ admits a factorization
\[
B \xrightarrow{\;\pi_B\;} B/{\sim} \xrightarrow{\;\bar{\iota}\;} A/\{a,b\},
\]
where:
\begin{itemize}
\item $\sim$ is the equivalence relation induced by $\pi \circ \iota$,
\item $\pi_B$ is a (possibly trivial) admissible elementary contraction,
\item $\bar{\iota}$ is an embedding.
\end{itemize}
Moreover, the factorization is unique up to unique isomorphism.
\end{lemma}

\begin{proof}
\textbf{Step 1: Description of the quotient on $B$.} Define an equivalence relation on $B$ by
\[
u \sim v \quad\Longleftrightarrow \quad (\pi\circ\iota)(u) = (\pi\circ\iota)(v).
\]
Since $\pi$ identifies only $a$ and $b$, and $\iota$ is injective, there are two possibilities.

\begin{enumerate}
\item $\{a,b\} \not\subseteq \Ima(\iota)$. Then no two distinct elements of $B$ are identified, so $\sim$ is trivial and $\pi_B = \id_B$.

\item $\{a,b\} \subseteq \Ima(\iota)$. Let $x,y \in B$ be the unique elements satisfying $\iota(x)=a$ and $\iota(y)=b$. Since $\iota$ is an embedding and $a\lessdot b$, we have $x \lessdot y$ in $B$. Moreover, $\sim$ identifies exactly the pair $\{x,y\}$. Hence $B/{\sim} \cong B/\{x,y\}$, and $\pi_B$ is the elementary contraction along the edge $x \lessdot y$. By Lemma~\ref{admissibility}, the edge $x \lessdot y$ is admissible in $B$, so $\pi_B$ is an admissible elementary contraction.
\end{enumerate}

\medskip

\textbf{Step 2: Construction of $\bar{\iota}$.}
Since $\pi_B$ is the quotient by the kernel relation of $\pi \circ \iota$, the universal property yields a unique map
\[
\bar{\iota}: B/{\sim} \longrightarrow A/\{a,b\}
\]
such that $\bar{\iota} \circ \pi_B = \pi \circ \iota$. We claim that $\bar{\iota}$ is an embedding. It is clearly injective, so it remains to show that it is order-preserving and order-reflecting.

Firstly we show that $\bar{\iota}$ is order-preserving. For Case (1) in Step 1, one has $\bar{\iota} = \pi \circ \iota$, which is clearly order-preserving. For Case (2) in Step (1), we know that $B/\sim = B/\{x, y\}$. Take $u, v \in B$ such that $[u] \leqslant [v]$ in $B/\{x, y\}$. We need to check
\[
\bar{\iota}([u]) = (\pi \circ \iota)(u) \leqslant (\pi \circ \iota)(v) = \bar{\iota}([v]).
\]
 Otherwise, by Lemma \ref{comparison lemma}, there are two cases:
\begin{enumerate}
\item $u \leqslant v$ in $B$. Then the conclusion holds trivially since both $\pi$ and $\iota$ are order-preserving.

\item $u = y$ and $x \leqslant v$. Then one has
\[
(\pi \circ \iota)(u) = (\pi \circ \iota)(y) = \pi(b) = \pi(a) = (\pi \circ \iota)(x) \leqslant (\pi \circ \iota)(v)
\]
in $A/\{a, b \}$.
\end{enumerate}
Thus $\bar{\iota}$ is order-preserving.

Now we check that $\bar{\iota}$ is order-reflecting. Suppose $\bar{\iota}([u]) \leqslant \bar{\iota}([v])$, or equivalently $\pi(\iota(u)) \leqslant \pi(\iota(v))$. By Lemma~\ref{comparison lemma}, one has two cases:
\begin{enumerate}
\item $\iota(u)\leqslant \iota(v)$. In this case, since $\iota$ is order-reflecting, one has $u \leqslant v$, so $[u] \leqslant [v]$ in $B/\sim$.

\item $\iota(u) = b$ and $a \leqslant \iota(v)$. Because $a \lessdot b$ is admissible, either $a$ is the root or $b$ is the unique immediate successor of $a$. In the first case, since $\bar{\iota}$ maps $[u] \in B/\sim$ to the root $[b] \in A/\{a, b\}$, it follows that $[u]$ is the root in $B/\sim$, so clearly one has $[u] \leqslant [v]$ in $B/\sim$. In the second case, if $a = \iota(v)$, then $[u] = [v]$ in $B/\sim$; if $a < \iota(v)$, since $b$ is the unique child of $a$ in $A$, we deduce $\iota(u) = b \leqslant \iota(v)$ in $A$. Consequently, $u \leqslant v$ in $B$ since $\iota$ is order-reflecting, so $[u] \leqslant [v]$ in $B/\sim$ again.
\end{enumerate}
Thus $\bar{\iota}$ is order-reflecting, and therefore an embedding.

\medskip
\textbf{Step 3: Uniqueness.}
The relation $\sim$ is precisely the kernel relation of $\pi\circ\iota$. Therefore any factorization of $\pi \circ \iota$ through a quotient of $B$ must factor uniquely through $B/{\sim}$ by the universal property of the quotient. Hence the above factorization is unique up to unique isomorphism.
\end{proof}

The above result fails for arbitrary elementary contractions, illustrated by the following example.

\begin{example} \label{admissible example}
Let $A = \{0 < 1, \, 1 < 2, \, 1 < 3\}$ and $B = \{0<2, \; 0<3\}$. Then the natural inclusion $\iota: B \hookrightarrow A$ is an embedding. Consider the elementary contraction
\[
\pi: A \twoheadrightarrow A/\{1,2\}
\]
along the edge $1 \lessdot 2$. This edge is not admissible, since $1$ has two immediate successors. Since $1 \notin \Ima(\iota)$ and $2 \in \Ima(\iota)$, the equivalence relation induced by $\pi \circ\iota$ on $B$ is trivial. Hence the only possible choice is $\pi_B = \id_B$.

The quotient $A/\{1,2\}$ is the chain $0 < [1,2] <3$. Therefore the induced map
\[
\bar{\iota}: B \longrightarrow A/\{1,2\}
\]
sends the incomparable elements $2$ and $3$ of $B$ to comparable elements $[1,2]<3$. Thus $\bar{\iota}$ is not order-reflecting, and hence is not an embedding.
\end{example}

Let $\T$ be the category of finite rooted trees, whose morphisms are finite composites of embeddings (including isomorphisms) and admissible elementary contractions. Define $\T^+$ to be the wide subcategory of $\T$ consisting of embeddings, and $\T^-$ the subcategory of isomorphisms and \emph{admissible contractions}, which are finite composites of admissible elementary contractions. Define a degree map
\[
d: \Ob(\T) \to \mathbb{N}, \quad A \mapsto |A|.
\]
Although it is not a small category, this does not cause much trouble to us since we can replace it by a full skeleton if necessary.

\begin{proposition} \label{poset is Reedy}
The quadruple $(\T,\T^+,\T^-,d)$ is a generalized Reedy structure.
\end{proposition}

\begin{proof}
We verify the axioms of a generalized Reedy category in Definition \ref{def of Reedy}. The first three axioms clearly hold, so we only need to check the last one about factorization of morphisms.

We first show that every morphism in $\T$ factors as an admissible contraction followed by an embedding. By definition, every morphism is a finite composite of embeddings and admissible elementary contractions. It therefore suffices to show that an embedding followed by an admissible elementary contraction can be rewritten in the opposite order. This is exactly the conclusion of Lemma \ref{exchange}.

It remains to show uniqueness of the factorization up to isomorphisms. Suppose
\[
f=f^+\circ f^- = g^+\circ g^-
\]
is a morphism $f \in \T(A, B)$ with $f^+,g^+$ morphisms in $\T^+$ and $f^-,g^-$ morphisms in $\T^-$. Since $f^+$ and $g^+$ are embeddings, they are injective. Hence for any $x, y$ in the domain,
\[
f^-(x) = f^-(y) \iff f(x) = f(y) \iff g^-(x) = g^-(y).
\]
Therefore $f^-$ and $g^-$ induce the same equivalence relation on the domain, and both codomains are naturally identified with the quotient by this equivalence relation. Hence there exists a unique bijection $\sigma: \mathrm{cod}(f^-) \to \mathrm{cod}(g^-)$ such that the following diagram commutes:
\[
\xymatrix{
 & \mathrm{cod}(f^-) \ar[rd]^-{f^+} \ar@{-->}[dd]^-{\sigma} & \\
A \ar[rd]_-{g^-} \ar[ru]^-{f^-} & & B\\
 & \mathrm{cod}(g^-) \ar[ru]_-{g^+} &
}
\]
Note that $f^+$ can be viewed as an isomorphism from $\mathrm{cod}(f^-)$ to $f(A)$, and $g^+$ is an isomorphism from $\mathrm{cod}(g^-)$ to $f(A)$. Since $g^+ \circ \sigma = f^+$, it follows that $\sigma$ is also an isomorphism. This finishes the proof.
\end{proof}

\begin{remark}
Example \ref{admissible example} illustrates why the morphisms of $\T$ are defined using \emph{embeddings} and \emph{admissible elementary contractions}. Indeed, if one replaces embeddings by arbitrary order-preserving injections, then the map $\bar{\iota}$ in the example would be allowed as a positive morphism, but the exchange argument would require a different formulation. On the other hand, if we allow arbitrary elementary contractions, then Lemma \ref{exchange} fails, so the present Reedy factorization argument no longer applies. Thus the combination of embeddings and admissible elementary contractions is not merely a technical convenience: it ensures the generalized Reedy structure established above for $\T$.
\end{remark}

As we mentioned before, the category $\T$ is different from the dendroidal category studied in \cite{HM, GLW}, as illustrated by the following elementary example.

\begin{example}\label{ex:tree-dendroidal}
Let $C_2$ and $C_3$ be the binary and ternary corollas in the dendroidal category $\Omega$, respectively. Regard their sets of edges as rooted posets, ordered from the output edge towards the input edges. The corresponding objects of $\T$ are
\[
\begin{aligned}
P_2 & = \{r,a,b\}, && r < a, \quad r < b,\\
P_3&=\{r',a',b',c'\},&& r'<a',\quad r'<b',\quad r'<c'.
\end{aligned}
\]
There is a natural embedding in $\T^+(P_2,P_3)$:
\[
\iota: P_2 \longrightarrow P_3, \qquad r \mapsto r', \quad a \mapsto a', \quad b \mapsto b',
\]
However, there is no morphism $C_2\to C_3$ in the dendroidal category $\Omega$. Indeed, a dendroidal morphism is a morphism of the corresponding free colored operads. The operad generated by $C_2$ has a binary generating operation, whereas the operad generated by $C_3$ has only unary identity operations and a ternary generating operation (together with its permutations), and hence no binary operations. An operad morphism must preserve arities, so no such morphism exists.
\end{example}

\subsection{Normalization} \label{D-K correspondence}

In this subsection, we extend the normalization procedure from the simplex category to the much larger category $\T$.

\begin{lemma} \label{section}
Let $\pi: A \to A/\{a,b\}$ be the elementary contraction along an edge $a \lessdot b$. Define a map $s: A/\{a,b\} \to A$ by sending the equivalence class $[a]=[b]$ to $a$, and sending every other element to itself. Then $s$ is an embedding of posets. Moreover, $\pi \circ s = \id_{A/\{a,b\}}$.
\end{lemma}

\begin{proof}
The map $s$ is clearly injective, and by construction one has $\pi \circ s = \id_{A/\{a,b\}}$ as maps. It is order-reflecting: if $s(x) \leqslant s(y)$ in $A$, then applying $\pi$ gives
\[
x = \pi(s(x)) \leqslant \pi(s(y)) = y.
\]

It remains to show that $s$ is order-preserving. Take $x \leqslant y$ in $A/\{a,b\}$. We have
\[
\pi(s(x)) \leqslant \pi(s(y)).
\]
If $x = y$, there is nothing to prove. Thus we assume $\pi(s(x)) = x \neq y = \pi(s(y))$. Applying Lemma~\ref{comparison lemma} to $s(x)$ and $s(y)$, we obtain that either
\begin{enumerate}
\item $s(x)\leqslant s(y)$, or

\item $s(x)=b$ and $a\leqslant s(y)$.
\end{enumerate}
The second case is impossible since $b\notin \mathrm{Im}(s)$. Hence $s(x)\leqslant s(y)$.
\end{proof}

\begin{remark}
Let $\pi: A \to A/\{a,b\}$ be any elementary contraction (not necessarily admissible). Then there are two natural set-theoretic sections of $\pi$, obtained by sending the equivalence class $[a]=[b]$ either to $a$ or to $b$. However, only the section sending $[a]=[b]$ to the lower element $a$ is an embedding of posets in general. The alternative choice need not be order-preserving. Thus every elementary contraction admits a canonical section which is an embedding, but this choice is asymmetric.
\end{remark}

For any admissible elementary contraction $\pi$ and its canonical section $s$,
\[
(s \circ \pi) \circ (s \circ \pi) = s \circ (\pi \circ s) \circ \pi = s \circ \pi
\]
is an idempotent in $k\T(A, A)$.

\begin{lemma} \label{lem:commute}
Let $A$ be a finite rooted tree, $e$ and $f$ distinct admissible edges in $A$, and
\[
\epsilon_e = s_e \circ \pi_e, \qquad \epsilon_f = s_f \circ \pi_f
\]
be the associated idempotent in $k\T(A,A)$. If $e$ and $f$ are either disjoint or share only a common parent vertex, then $\epsilon_e \epsilon_f = \epsilon_f \epsilon_e$.
\end{lemma}

\begin{proof}
Given an edge $g$ connecting a parent vertex $u$ to a child vertex $v$, by construction, the endomorphism $\epsilon_g = s_g\pi_g$ acts on vertices by
\[
\epsilon_g(z)
=
\begin{cases}
u & \text{if } z=v,\\
z & \text{otherwise}.
\end{cases}
\]

We now use this definition to compare $\epsilon_e \epsilon_f$ and $\epsilon_f\epsilon_e$.

\medskip
\noindent
\textbf{Case 1: disjoint edges.} Let $e = (u_1 \lessdot v_1)$ and $f = (u_2 \lessdot v_2)$ with $\{u_1, v_1\} \cap \{u_2, v_2\} = \varnothing$. Then $\epsilon_e$ only changes the vertex $v_1$, while $\epsilon_f$ only changes the vertex $v_2$. Hence the two transformations act independently and therefore commute.

\medskip

\noindent
\textbf{Case 2: sibling edges.} Let $e = (a \lessdot b)$ and $f = (a \lessdot c)$ with $b \neq c$. Then
\[
\epsilon_e(b)=a, \qquad \epsilon_f(c)=a,
\]
and both maps fix every other vertex. Thus $\epsilon_e$ changes only $b$, while $\epsilon_f$ changes only $c$. Therefore applying them in either order produces the same result on every vertex of $A$.

\medskip
Since morphisms in $\T$ are determined by their underlying maps of posets, the conclusion follows.
\end{proof}

\begin{lemma} \label{lem:nested}
Let $e = (a \lessdot b)$ and $f = (b \lessdot c)$ be two nested admissible edges in $A$. Then one has
\[
\epsilon_e \epsilon_f \epsilon_e = \epsilon_e \epsilon_f = \epsilon_f \epsilon_e \epsilon_f.
\]
\end{lemma}

\begin{proof}
View $\epsilon_e$ and $\epsilon_f$ as endomorphisms of the set $V(A)$. By construction,
\[
\epsilon_e(z)=
\begin{cases}
a & \text{if } z=b,\\
z & \text{otherwise},
\end{cases}
\qquad
\epsilon_f(z)=
\begin{cases}
b & \text{if } z=c,\\
z & \text{otherwise}.
\end{cases}
\]

By a direct computation, we have:
\[
(\epsilon_e \epsilon_f) (z)=
\begin{cases}
a & \text{if } z \in \{b, c\},\\
z & \text{otherwise}.
\end{cases}
\]
and
\[
(\epsilon_e \epsilon_f \epsilon_e)(z)
=
\begin{cases}
a & z \in \{b, c\},\\
z & \text{otherwise},
\end{cases}
\]
so $\epsilon_e \epsilon_f = \epsilon_e \epsilon_f \epsilon_e$. One can check that the same formula holds for $\epsilon_f \epsilon_e \epsilon_f$, so $\epsilon_e \epsilon_f = \epsilon_f \epsilon_e \epsilon_f$ as well.
\end{proof}

Now we define the normalization operator for each rooted tree $A$. Given a vertex $v$ in $A$, let $\dep(v)$ be its depth. For an admissible edge $e = (u \lessdot v)$, define its depth $\dep(e) = \dep(v)$. Let
\[
E_r = \{e \in E_{\mathrm{ad}}(A) \mid \dep(e) = r\}.
\]
Then we can find a positive integer $m$ such that
\[
E_{\mathrm{ad}} = \bigsqcup_{r=1}^m E_r.
\]
For each stratum, define
\[
\Psi_{(r)} = \prod_{e \in E_r} (1 - \epsilon_e),
\]
which is well-defined by Lemma \ref{lem:commute}. In particular, when $E_r = \emptyset$, we set $\Psi_{(r)} = 1$. Define
\[
\Psi_A = \Psi_{(m)} \Psi_{(m-1)} \cdots \Psi_{(1)}.
\]

\begin{proposition}
Let $A$ be a finite rooted tree. Then $\epsilon_g \Psi_A = 0$ for every $g \in E_{\mathrm{ad}}(A)$. Moreover, $\Psi_A$ is an idempotent in $k\T(A, A)$.
\end{proposition}

\begin{proof}
Suppose that $g = (a \lessdot b)$ is contained in $E_r$. If $j > r+1$ and $f \in E_j$, then the endpoints of $f$ lie strictly above the endpoints of $g$, so $f$ cannot be nested with $g$. Hence $f$ and $g$ are either disjoint or share only a common parent vertex, and Lemma~\ref{lem:commute} applies. Therefore $\epsilon_g$ commutes with every factor in $\Psi_{(m)},\dots,\Psi_{(r+2)}$. Consequently, we may rewrite
\[
\epsilon_g\Psi_A = \Psi_{(m)}\cdots\Psi_{(r+2)} \cdot \epsilon_g \Psi_{(r+1)} \Psi_{{r}} \ldots \Psi_{(1)}.
\]

The factors in $\Psi_{(r)}$ other than $(1-\epsilon_g)$ commute with $\epsilon_g$, since any two distinct edges in the same stratum have the same depth and therefore cannot form a nested pair. Hence, by Lemma \ref{lem:commute}, we may move these factors across $\epsilon_g$. Similarly, among the factors in $\Psi_{(r+1)}$, the only ones that do not commute with $\epsilon_g$ are those corresponding to child edges of $g$. Denote this set of factors by $C(g)$. Therefore we may write
\[
\epsilon_g \Psi_A
=
\Psi_{(m)}\cdots\Psi_{(r+2)}
\cdot
\epsilon_g \Psi_{C(g)}(1-\epsilon_g)\Psi_{\mathrm{rest}},
\]
where $\Psi_{\mathrm{rest}}$ denotes the product of all remaining factors in the lower strata and the remaining factors in $E_{r+1}$.

Since edges in $C(g)$ are sibling edges, their idempotents commute. Therefore
\[
\Psi_{C(g)} = \prod_{f \in C(g)} (1 - \epsilon_f) = \sum_{S \subseteq C(g)} (-1)^{|S|} \epsilon_S,
\]
where
\[
\epsilon_S = \prod_{f \in S} \epsilon_f.
\]
Consequently, one has
\[
\epsilon_g \Psi_{C(g)} (1-\epsilon_g) = \sum_{S \subseteq C(g)} (-1)^{|S|} \epsilon_g \epsilon_S (1-\epsilon_g).
\]

Note that for every $f \in C(g)$, the edges $g$ and $f$ form a nested pair $a \lessdot b \lessdot c$. By Lemma~\ref{lem:nested}, one has $\epsilon_g \epsilon_f \epsilon_g = \epsilon_g \epsilon_f$. Since all elements of $C(g)$ commute with each other, repeated application gives
\[
\epsilon_g \epsilon_S \epsilon_g = \epsilon_g \epsilon_S \qquad \forall S \subseteq C(g).
\]
Therefore
\[
\epsilon_g \epsilon_S (1-\epsilon_g) = \epsilon_g \epsilon_S - \epsilon_g \epsilon_S \epsilon_g = 0.
\]
It follows that
\[
\epsilon_g \Psi_{C(g)} (1-\epsilon_g) = 0,
\]
and consequently $\epsilon_g \Psi_A = 0$ as desired.

The second statement follows readily from the first one. Indeed, since $\epsilon_g \Psi_A = 0$ for all $g \in E_{\mathrm{ad}}(A)$, we have
\[
(1 - \epsilon_g) \Psi_A = \Psi_A.
\]
Since the product defining $\Psi_A$ is ordered by increasing strata, we may successively apply the above identity for each $g \in E_{\mathrm{ad}}(A)$ without altering the remaining factors. This yields $\Psi_A^2 = \Psi_A$.
\end{proof}

As a consequence, we show that every standard module is projective.

\begin{theorem} \label{thm:projective_standard}
Let $A$ be a finite rooted tree and let $P_A = k\T(A,-)$. Then
\[
P_A = P_A \Psi_A \oplus \mathfrak{I}_A,
\]
where $\mathfrak{I}_A$ is the submodule generated by morphisms factoring through admissible elementary contractions.

Consequently, $\Delta_A$ is isomorphic to $P_A \Psi_A$, and hence is projective.
\end{theorem}

\begin{proof}
Since $\Psi_A$ is an idempotent in $k\T(A,A)$, we obtain the standard decomposition of left $\T$-modules
\[
P_A = P_A \Psi_A \oplus P_A(1-\Psi_A).
\]
We shall prove that $\mathfrak{I}_A = P_A(1-\Psi_A)$.

Let $g \in E_{\mathrm{ad}}(A)$ and let $\pi_g$ be the corresponded admissible elementary contraction. By construction of the section $s_g$, one has
\[
\pi_g = (\pi_g \circ s_g) \circ \pi_g = \pi_g \circ (s_g \circ \pi_g) = \pi_g \circ \epsilon_g.
\]
Consequently,
\[
\pi_g \Psi_A = \pi_g \epsilon_g \Psi_A = 0.
\]
Therefore every morphism factoring through an admissible elementary contraction is annihilated by $\Psi_A$, and hence $\mathfrak{I}_A \subseteq P_A(1-\Psi_A)$.

On the other hand, note that
\[
\Psi_A = \prod_{e \in E_{\mathrm{ad}}(A)} (1-\epsilon_e),
\]
where the product is taken in the fixed order of strata. Expanding this ordered product, every term in $\Psi_A$ is a composition of idempotents $\epsilon_e$, and every term different from the identity contains at least one factor $\epsilon_e$ for some $e \in E_{\mathrm{ad}}(A)$. Hence
\[
1 - \Psi_A \in \sum_{e \in E_{\mathrm{ad}}(A)} k\T(A,A)\,\epsilon_e.
\]
Since $\epsilon_e = s_e \circ \pi_e$ and $\pi_e$ is contained in $\mathfrak{I}_A$, it follows that each $k\T(A,A)\,\epsilon_e$ and hence the above sum are contained in $\mathfrak{I}_A$, so is $1 - \Psi_A$. Consequently,  $P_A(1-\Psi_A) \subseteq \mathfrak{I}_A$. Combining both inclusions gives $\mathfrak{I}_A = P_A(1-\Psi_A)$.

We have shown
\[
P_A = P_A\Psi_A \oplus P_A(1-\Psi_A) = P_A \Psi_A \oplus \mathfrak{I}_A.
\]
Thus $\Delta_A = P_A/\mathfrak{I}_A \cong P_A\Psi_A$, which is projective.
\end{proof}

As immediate consequences, we have:

\begin{corollary} \label{tree projective generators}
Let $k$ be a commutative ring. Then standard modules form a family of projective generators of $\T \Mod$.
\end{corollary}

\begin{proof}
Every negative morphism $A\to B$ is surjective, so its stabilizer under postcomposition by $G_B$ is trivial. Thus $G_B$ acts freely on $\T^-(A,B)$ for all $A, B \in \Ob(\T)$. Consequently, the standard-filtration argument of \cite[Proposition 3.11]{DLL} gives the following conclusion: each representable $P_A = k\T(A, -)$ has a finite filtration whose factors are isomorphic to finite direct sums of standard modules. But since each standard module is projective, it follows that $P_A$ is a finite direct sum of standard modules. The conclusion follows from this observation.
\end{proof}

\begin{corollary}
Let $\mathscr{E}$ be the $k$-linear category whose objects are finite rooted trees and whose morphism spaces are
\[
\mathscr{E}(B, A) = \Hom_{k\T}(\Delta_A,\Delta_B).
\]
Then $\T \Mod$ is Morita equivalent to $\mathscr{E} \Mod$. Moreover, $\mathscr{E}$ is a directed category. Consequently, irreducible $\T$-modules are in bijection with pairs $(A, S)$, where $A$ is a finite rooted tree and $S$ is a simple $k\Aut(A)$-module.
\end{corollary}

\begin{proof}
The first statement follows from the previous corollary and the Morita equivalence; the second one follows from Proposition \ref{semi-orthogonal}; and the third one is clear since $\mathscr{E}$ is directed.
\end{proof}

\subsection{Dold--Kan correspondence} \label{uniform D-K correspondence}

A well known fact for the simplex category is that nontrivial homomorphisms between standard modules occur only in adjacent degrees. In this subsection we generalize it to many subcategories of $\T$.

In the rest of this paper let $\D$ be a subcategory of $\T$ together with a distinguished class of objects $\D_{\min} \subseteq \D$ (called \emph{minimal objects}), satisfying the following two conditions:
\begin{enumerate} \label{C1}
\item[(C0)] For every $B \in \D_{\min}$, every embedding $\iota: B \hookrightarrow A$ with $A \notin\D_{\min}$, and every admissible elementary contraction $\pi: A \to A/e$, the composite $\pi \circ \iota$ is an embedding.
\item[(C1)] if $A$ is minimal, then it admits no proper embedding $B \hookrightarrow A$ in $\D$; if $A$ is not minimal, and $a \lessdot b$ is admissible in $A$, then the quotient $A/\{a,b\}$ is isomorphic to an object in $\D$.
\end{enumerate}
Intuitively, (C0) says that no admissible edge of a non-minimal object $A$ is entirely contained in the image of an embedding of a minimal object, and (C1) is a contraction-closure condition away from minimal objects. Morphisms in $\D$ are finite composites of embeddings and admissible elementary contractions, with the restriction that admissible contractions may only be applied to non-minimal objects. We will see that the class $\D_{\min}$ plays a crucial role in Remark \ref{minimal objects}.

\begin{proposition} \label{Reedy of D}
The subcategory $\D$ is a generalized Reedy category with positive morphisms the embeddings and negative morphisms generated by admissible elementary contractions starting from non-minimal objects. Moreover, every standard module is projective.
\end{proposition}

\begin{proof}
We first verify the Reedy factorization. Since the morphisms of $\D$ are generated by embeddings and the allowed elementary contractions, it is enough to exchange two consecutive morphisms
\[
B \xrightarrow{\iota} A \xrightarrow{\pi} A/e,
\]
where $\iota$ is an embedding and $\pi$ is an admissible elementary contraction with $A$ non-minimal. By Lemma~\ref{exchange}, the composite has a factorization $\pi \circ \iota = \bar\iota \circ \pi_B$, where $\bar\iota$ is an embedding and $\pi_B$ is either the identity or an admissible elementary contraction of $B$. We have two cases:
\begin{itemize}
\item If $B$ is non-minimal, Condition~(C1) ensures that the target of $\pi_B$ belongs to $\D$, so both factors are morphisms of $\D$.

\item If $B$ is minimal, Condition~(C0) ensures that $\pi\circ\iota$ is already an embedding, and we may take $\pi_B = 1_B$.
\end{itemize}
Thus every such exchange stays in $\D$. Repeating these exchanges, we get the desired Reedy factorization for morphisms in $\D$. Uniqueness up to isomorphism follows from that in $\T$, since the factors belong to $\T^-$ and $\T^+$ and all isomorphisms between objects of $\D$ belong to $\D$.

We next prove projectivity of standard modules. If $A$ is minimal, there are no noninvertible negative morphisms starting from $A$, so $\mathfrak I_A^{\D}=0$ and $\Delta_A^{\D}=P_A^{\D}$ is projective. In this case we set $\Psi_A^{\D}=1_A$. When $A$ is not minimal, note that the normalization construction of Subsection \ref{D-K correspondence} uses only admissible elementary contractions and their canonical sections. Thus by condition {\rm(C1)}, every admissible elementary contraction starting at $A$ has target again in $\D$, and Lemma \ref{section} shows that the associated section is an embedding. Consequently, all idempotents $\epsilon_e = s_e \circ \pi_e$ belong to $k\D(A,A)$. The normalization argument therefore applies verbatim in $\D$, yielding a decomposition
\[
P_A^{\D} = P_A^{\D} \Psi_A^{\D} \oplus \mathfrak{I}_A^{\D}
\]
for every object $A$ of $\D$. Consequently every standard module in $\D$ is projective.
\end{proof}

Now we introduce a combinatorial notation playing a crucial role in this subsection.

\begin{definition}
Given an object $A$ in $\D$, a vertex $a \in V(A)$ is called \emph{flexible} if there is a proper embedding $f: B \hookrightarrow A$ such that $a \notin f(B)$.
\end{definition}

Using this notion, we impose the second condition on $\D$.

\begin{enumerate} \label{C2}
\item[(C2)] For every non-minimal object $A$ in $\D$, flexible vertices are connected via admissible edges in $A$.
\end{enumerate}

We give two examples to illustrate these conditions.

\begin{example} \label{simplex category}
Let $\D$ be the full subcategory of $\T$ consisting of finite chains, which is equivalent to the simplex category. The only minimal object is the singleton set. In this case, every elementary contraction is admissible, coinciding with elementary degeneracy maps. Since embeddings are face maps, every vertex of a non-minimal chain is flexible. Thus (C0), (C1) and (C2) hold.
\end{example}

\begin{example} \label{type-D category}
For $n \geqslant 1$, let $D_n = \{1,\ldots,n\} \sqcup \{x, y\}$ with relations
\[
1 < 2 < \ldots < n, \quad n < x, \quad n < y.
\]
Let $\D$ be the subcategory of $\T$ generated by all embeddings between the $D_n$ and by admissible trunk contractions $D_n\to D_{n-1}$ for $n\geqslant 2$. The unique minimal object is $D_1$. For $n\geqslant2$, the admissible edges of $D_n$ are exactly its trunk edges; the admissible fork-edge contractions of $D_1$ in $\T$ are not included in $\D$. Thus (C1) holds. Moreover, every embedding $D_1\hookrightarrow D_n$ contains both fork vertices but only one trunk vertex in its image. No allowed trunk edge can therefore be contained in that image, verifying (C0). Moreover, for any embedding $\iota: D_m \to D_n$ with $m < n$, the two fork vertices in $D_n$ must be contained in the image of $\iota$, so flexible vertices are precisely those appearing in the trunk part. Thus (C2) follows.

Note that $\Aut(D_n) \cong C_2$, generated by the involution exchanging $x$ and $y$. Every morphism $f: D_m \to D_n$ either preserves the ordered pair $(x,y)$ or exchanges it. Restricting f to the trunk $1 < 2 < \ldots < m$ yields an order-preserving map of finite chains. Conversely, every order-preserving map of chains extends uniquely to a fork-preserving morphism $D_m \to D_n$, and composing with the involution yields a unique fork-reversing morphism. More precisely, $\D$ is equivalent to the category of nonempty finite chains with \emph{signed order-preserving maps}: a morphism is a pair $(u,\varepsilon)$, where $u$ is order-preserving and $\varepsilon\in C_2$, and composition multiplies the signs. The fork exchange does not reverse the order on the trunk.
\end{example}

For the general results that follow, we assume that $\D$ is the category with all embeddings described above, satisfying (C0), (C1), and (C2).

\begin{lemma} \label{two mutable vertices}
Let $A$ be an object of $\D$ admitting a proper embedding $B \hookrightarrow A$. Then $A$ has at least two flexible vertices.
\end{lemma}

\begin{proof}
Let $r$ be the root of $A$. Since $A$ admits a proper embedding, it is not minimal. Thus $r$ has children. If $r$ has at least two children $c_1$ and $c_2$, then by the proof of Lemma \ref{section} we get two sections $s_i: A/\{r, c_i\} \hookrightarrow A$ which are embeddings. It follows that both $c_1$ and $c_2$ are flexible.

Now suppose that $r$ has a unique child $c$. The above argument already shows that $c$ is flexible. Since $A$ is not minimal and the edge $r\lessdot c$ is admissible, condition {\rm(C1)} implies that $A/\{r,c\}$ is isomorphic to an object $B$ of $\D$. Since $c$ is the unique child of $r$, the quotient $A/\{r,c\}$ is naturally isomorphic to the rooted subtree of $A$ with root $c$. Hence this subtree may be regarded as an object of $\D$, and its natural inclusion into $A$ is a proper embedding whose image does not contain $r$. Therefore $r$ is flexible as well. Thus in this case $A$ also has at least two flexible vertices.
\end{proof}

Note that a morphism $\pi: A \to C$ in $\D^-$ is unfactorizable if and only if it is an admissible elementary contraction, namely there is an admissible edge $a \lessdot b$ in $A$ such that $C \cong A/\{a, b\}$, and $\pi$ is the composite of the natural projection $A \to A/\{a, b\}$ and the isomorphism. In particular, we have $|A| = |C| + 1$.

Given an embedding $f: B \to A$ in $\D^+$ and an allowed admissible elementary contraction $\pi: A \to C$ identifying $a, b \in A$ with $a \lessdot b$, the following result describes the structure of the associated fiber.

\begin{lemma}\label{fibers}
The following statements hold:
\begin{enumerate}
\item Every proper embedding $f: B \to A$ lies in a fiber of $\pi^{\ast}$ induced by an admissible elementary contraction $\pi: A \to A/\{a, b\}$.

\item If $f$ lies in a fiber of $\pi^{\ast}$ induced by $\pi: A \to A/\{a, b\}$, then $\{a, b\} \nsubseteq f(B)$; moreover, if $\{a, b\} \cap f(B) = \varnothing$, then the fiber containing $f$ is singleton.

\item Each fiber of $\pi^{\ast}$ has at most two embeddings.
\end{enumerate}
\end{lemma}

\begin{proof}
(1) Since $f$ is a proper embedding, there exists a vertex $a \in A \setminus f(B)$, hence $a$ is flexible. By Lemma \ref{two mutable vertices}, there exists another flexible vertex $b \neq a$ in $A$. By condition (C2), the flexible vertices are connected via admissible edges, so there exists a path of admissible edges connecting $a$ to $b$. The first edge connecting $a$ to another vertex $a'$ is then admissible and not contained in $f(B)$. We use it to define the elementary contraction
\[
\pi: A \twoheadrightarrow A/\{a,a'\}.
\]
Since $\{a,a'\} \nsubseteq f(B)$, the induced equivalence relation on $B$ is trivial. Hence, by the proof of Lemma \ref{exchange}, the composite $\pi \circ f$ is again an embedding, so $f$ lies in a fiber of $\pi^{\ast}$.

\medskip
(2) If $\pi \circ f$ is an embedding, then it is injective. Since $\pi$ identifies exactly $\{a,b\}$, it follows that $\{a,b\} \nsubseteq f(B)$. Moreover, if neither $a$ nor $b$ lies in $f(B)$, and $g: B \to A$ is another embedding in the same fiber, then $\pi \circ f = \pi \circ g$ implies $f=g$ since $\pi$ is injective on $A \setminus \{a,b\}$.

\medskip
(3) Let $g:B \to A$ be another embedding such that $\pi \circ f = \pi \circ g$ is an embedding. Then $f(B)$ contains at most one element in $\{a, b\}$. In the case that $\{a,b\} \cap f(B) = \varnothing$, one has $f = g$ by (2). Otherwise, without loss of generality assume that $f(B) \cap \{a, b\} = \{a \}$. Let $u \in B$ be the unique preimage of $a$. Since $\pi$ is injective on $A \setminus\{a,b\}$, any $g \neq f$ must satisfy $g(u) = b$ and agree with $f$ elsewhere. Hence the fiber contains at most two embeddings.
\end{proof}

Thus each fiber of $\pi^{\ast}$ induced by an admissible elementary contraction contains at most two embeddings, and in the case that it contains two embeddings, we say that one is the \emph{mutation} of the other.

\begin{remark} \label{minimal objects}
The distinguished minimal objects allow one to retain categories whose smallest objects are nontrivial rooted trees. For instance, in Example~\ref{type-D category}, the fork edges of $D_1$ are admissible in $\T$, but their contractions produce two-vertex chains outside the chosen family $\{D_n\}_{n\geqslant1}$. Condition~(C1) permits us to omit these contractions, while Condition~(C0) ensures that this omission is compatible with the exchange lemma.
\end{remark}

\begin{lemma} \label{singleton}
Let $f: B \to A$ be an embedding in $\D^+$ with $|A| - |B| \geqslant 2$. Then there exists an embedding $g: B \to A$ obtained from $f$ by finitely many mutations such that $g$ is contained in a singleton fiber.
\end{lemma}

\begin{proof}
Let $H(f) = A \setminus f(B)$ be the set of holes of $f$. Since $|A| - |B| \geqslant 2$, the set $H(f)$ contains at least two vertices. Note that all elements in $H(f)$ are flexible, and any two vertices in them are connected by a path whose edges are admissible in $A$ by (C2). Define
\[
d(f) = \min\{\, d(a, b) \mid a, b \in H(f),\ a \neq b\,\},
\]
where $d(a, b)$ denotes the length of a shortest path of admissible edges connecting $a$ and $b$. We proceed by induction on $d(f)$.

If $d(f) = 1$, then there exist adjacent holes $a, b \in A \setminus f(B)$ connected by an admissible edge. Let $\pi$ be the elementary contraction along this admissible edge. Since $\{a, b\} \cap f(B) = \varnothing$, Lemma~\ref{fibers} implies that $f$ is contained in a fiber of $\pi^\ast$ and this fiber is singleton. Hence we are done.

Suppose that $d(f) > 1$. Choose holes $a, b \in H(f)$ such that $d(a, b) = d(f)$ and let
\[
a = v_0 \sim v_1 \sim \cdots \sim v_\ell = b
\]
be a shortest admissible path between $a$ and $b$. By minimality of $d(f)$, every interior vertex $v_i$ for $0 < i < \ell$ lies in $f(B)$. Let $\pi$ be the elementary contraction along the first admissible edge $a \sim v_1$. By the proof of Lemma \ref{exchange}, the composite $\pi \circ f$ is again an embedding, hence $f$ lies in a fiber of $\pi^\ast$. If this fiber is singleton, we are done. Otherwise, by Lemma \ref{fibers}, there exists another embedding $g$ in the same fiber. Since $\pi$ identifies only $a$ and $v_1$, the embeddings $f$ and $g$ differ only at the unique vertex of $B$ mapping to $v_1$. In particular,
\[
a \in g(B), \quad v_1 \notin g(B),
\]
and all other image vertices are unchanged. Hence $v_1$ and $b$ are holes for $g$, but
\[
d(g) \leqslant d(v_1,b) < d(a,b) = d(f).
\]
By induction, after finitely many mutations we reach an embedding lying in a singleton fiber.
\end{proof}

The following result shows that nontrivial homomorphisms between standard modules only appear in adjacent degrees.

\begin{proposition} \label{thin 1}
Let $k$ be a commutative ring. Given $A, B \in \Ob(\D)$, let $\Delta_A, \Delta_B$ be the corresponded standard modules in $\D \Mod$. Then
\[
\Hom_{k\D} (\Delta_A, \Delta_B) \neq 0 \quad \Rightarrow \quad |A| = |B| \ \text{or}\ |A| = |B|+1.
\]
\end{proposition}

\begin{proof}
If $|A| < |B|$, then clearly $\Hom_{k\D} (\Delta_A, \Delta_B) = 0$ by Proposition~\ref{semi-orthogonal}. Suppose that $|A| \geqslant |B| + 2$. By Theorem \ref{hom spaces general}, a morphism $c \in \Hom_{k\D} (\Delta_A, \Delta_B)$ is equivalent to a system of scalars $(c_f)_{f \in \D^+(B, A)}$ satisfying the following compatibility relations: for every fiber $F \in \mathcal{F}_{B, A}$, one has
\[
\sum_{f \in F} c_f = 0.
\]
Note that each admissible edge $x \lessdot y$ in $A$ gives an allowed unfactorizable morphism in $\D^-(A,-)$, since $A$ is non-minimal, and every embedding $f: B \to A$ lies in at least one fiber induced by an admissible elementary contraction by Lemma \ref{fibers}. Moreover, every non-singleton fiber consists of exactly two mutations.

Let $f$ be an arbitrary embedding in $\D^+(B, A)$. By Lemma \ref{singleton}, we obtain a finite mutation sequence $f = f_0, f_1, \ldots, f_n = g$ such that each consecutive pair $f_i, f_{i+1}$ form a fiber of $\pi^{\ast}$ induced by some admissible elementary contraction $\pi$, and $g$ forms a singleton fiber of a certain $\pi^{\ast}$. Thus one has $c_{f_i} + c_{f_{i+1}} = 0$ and $c_{f_n} = 0$. It follows that $c = 0$.
\end{proof}

In the rest of this paper we focus on the critical case $|A| = |B| + 1$.

\begin{definition}
Given $A \in \Ob(\D)$, a vertex $a \in A$ is called \emph{mutable} if there is a certain $B \in \Ob(\D)$ with $|B| = |A| - 1$ as well as an embedding $f: B \hookrightarrow A$ such that $A \setminus f(B) = \{a\}$.

The \emph{mutation graph} $\Sigma_A$ is defined as follows: vertices in $\Sigma_A$ are mutable elements in $A$, and edges in $\Sigma_A$ are admissible edges in $A$ whose both endpoints are mutable.
\end{definition}

Clearly, mutable vertices in $A$ are flexible. One can check that the two flexible vertices constructed in the proof of Lemma \ref{two mutable vertices} are mutable. Moreover, all flexible vertices in Examples \ref{simplex category} and \ref{type-D category} are mutable.

\begin{lemma} \label{multiplicity one}
Let $A,B$ be objects in $\D$ with $|A|=|B|+1$. For an embedding $f: B \to A$, let $h(f)$ be the unique element of $A \setminus f(B)$. If $f$ and $g$ are connected by a sequence of mutations and $h(f) = h(g)$, then $f=g$.
\end{lemma}

\begin{proof}
Each mutation replaces $f_i$ by $f_{i+1}$ by contracting an admissible edge, and hence the unique hole moves along that edge. Thus a sequence of mutations
\[
f = f_0, f_1, \dots, f_n = g
\]
induces a walk
\[
h(f_0) \to h(f_1) \to \cdots \to h(f_n)
\]
in the underlying tree of $A$.

Assume the sequence has minimal length among all sequences connecting $f$ and $g$. Then the corresponding walk has no backtracking. Indeed, if a subwalk
\[
v_i \to v_{i+1} \to v_i
\]
occurred, then the two consecutive mutations are performed along the same admissible edge. Since a mutation exchanges the two embeddings in the corresponding non-singleton fiber, the two mutations are inverse to each other, and this part of the sequence can be removed, contradicting the minimality of the length. Since $h(f_0) = h(f_n)$, this is a reduced closed walk in a tree. Therefore it must be trivial, so $n = 0$ and $f=g$.
\end{proof}

By this lemma, when $|A| = |B| + 1$, for any fixed connected component of the fiber graph $\Gamma_{B, A}$, different embeddings $f, g: B \hookrightarrow A$ lying in this component cannot have the same image. Moreover, by the following result, connected components of $\Gamma_{B, A}$ are isomorphic to connected components of the mutation graph $\Sigma_A$.

\begin{proposition} \label{subgraph}
Let $f: B \to A$ be an embedding with $|A| = |B| + 1$. Then the connected component $\Gamma_f$ containing $f$ of $\Gamma_{B, A}$ is isomorphic to the connected component of $\Sigma_A$ containing the unique hole of $A \setminus f(B)$.
\end{proposition}

\begin{proof}
For an embedding $g: B \hookrightarrow A$ in $\Gamma_f$, let $x_g$ be the unique element of $A \setminus g(B)$, which is mutable and hence contained in $\Sigma_A$. We claim that the assignment
\[
\Phi: \Gamma_f \longrightarrow \Sigma_A, \qquad g \longmapsto x_g,
\]
induces an isomorphism from $\Gamma_f$ onto the connected component of $\Sigma_A$ containing $x_f$.

\medskip
We first show that $\Phi$ is injective. Suppose $g, h \in \Gamma_f$ satisfy $\Phi(g) = \Phi(h)$. Then $x_g = x_h$. Since $g$ and $h$ belong to the same connected component $\Gamma_f$, they are connected by a sequence of mutations. Lemma \ref{multiplicity one} therefore implies that $g = h$. Hence $\Phi$ is injective.

\medskip
Next we show that $\Phi$ preserves adjacency. Let $g, h \in \Gamma_f$ be adjacent vertices. By definition, $g$ and $h$ belong to the same fiber of
\[
\pi^\ast: \D^+(B, A) \longrightarrow \D^+(B, A/\{a, b\})
\]
for some admissible elementary contraction $\pi: A \twoheadrightarrow A/\{a, b\}$. Since $g \neq h$, (3) of Lemma \ref{fibers} implies that one of the embeddings omits $a$ and the other omits $b$. Consequently, $\{x_g, x_h\} = \{a, b\}$. Since $a \lessdot b$ is admissible and both $a$ and $b$ occur as holes of embeddings, they are mutable vertices. Therefore $\{a, b\}$ is an edge of $\Sigma_A$, and hence $\Phi(g)$ and $\Phi(h)$ are adjacent in $\Sigma_A$. Thus $\Phi$ is a graph homomorphism.

\medskip

We now prove the converse, namely that $\Phi$ reflects the adjacency relation. Let $g \in \Gamma_f$, and let $y$ be a vertex of $\Sigma_A$ adjacent to $\Phi(g) = x_g$. Let $a \lessdot b$ be the admissible edge joining $x_g$ and $y$, and let $\pi: A \twoheadrightarrow A/\{a, b\}$ be the corresponded elementary contraction. Since $y$ is mutable, there exists an embedding $h: C \hookrightarrow A$ with $|C| = |A| - 1$ such that $x_h = y$.

Since $x_g$ and $x_h$ are the unique holes of $g$ and $h$ respectively, one has $g(B) = A \setminus \{x_g\}$ and $h(C) = A \setminus \{y\}$. Because $\{x_g, y\} = \{a, b\}$, both images contain exactly one endpoint of the contracted edge. Hence the induced equivalence relation on $B$ is trivial for both embeddings. By the proof of Lemma \ref{exchange}, the composites
\[
\pi \circ g: B \to A/\{a, b\}, \quad \pi \circ h: C \to A/\{a, b\}
\]
are embeddings. Since $B$ and $C$ are both finite sets of the same cardinality as $A/\{a,b\}$, $\pi \circ g$ and $\pi \circ h$ are actually isomorphisms of posets. Hence $B \cong C$, and after fixing an identification we may assume $C = B$.

Under this identification, $\pi \circ g$ and $\pi \circ h$ are isomorphisms from $B$ to $A/\{a,b\}$. Thus there exists $\sigma \in \Aut(B)$ such that
\[
\pi \circ g = \pi \circ (h \circ \sigma).
\]
Hence $g$ and $h \circ \sigma$ lie in the same fiber of $\pi^\ast$. These two embeddings are distinct, and hence are adjacent in $\Gamma_{B, A}$. Since $\Phi(g) = x_g$ and $\Phi(h \circ \sigma) = y$, it follows that every neighbor $y$ of $x_g$ in $\Sigma_A$ is the image under $\Phi$ of a neighbor of $g$ in $\Gamma_f$.

\medskip
We have shown that for every vertex $g \in \Gamma_f$, the map $\Phi$ induces a bijection between the neighbors of $g$ in $\Gamma_f$ and the neighbors of $\Phi(g)$ in $\Sigma_A$. Consequently, the image of $\Phi$ is a connected subgraph of $\Sigma_A$ containing $x_f$ and closed under taking neighbors. Hence the image of $\Phi$ is precisely the connected component of $\Sigma_A$ containing $x_f$. Since $\Phi$ is injective, it follows that $\Phi$ is a graph
isomorphism from $\Gamma_f$ onto that connected component.
\end{proof}

\begin{remark}
Although every connected component of $\Gamma_{B, A}$ is isomorphic to a connected component of $\Sigma_A$, and every connected component of $\Sigma_A$ arises in this way, there is generally no bijection between the connected components of the two graphs. Distinct connected components of $\Gamma_{B, A}$ may be isomorphic to the same connected component of $\Sigma_A$. Consequently, $\Gamma_{B, A}$ may contain several disjoint copies of a given connected component of $\Sigma_A$.
\end{remark}

\begin{proposition} \label{thin 2}
Let $k$ be a commutative ring, and let $A,B$ be objects in $\D$ with $|A| = |B| + 1$. Then $\Hom_{k\D} (\Delta_A, \Delta_B)$ is a free $k$-module whose rank equals the number of connected components of $\Gamma_{B,A}$ containing no singleton fiber.
\end{proposition}

\begin{proof}
By Theorem \ref{hom spaces general}, an element of $\Hom_{k\D}(\Delta_A, \Delta_B)$ is identified with a system of coefficients
\[
(c_f)_{f \in \D^+(B,A)}
\]
satisfying the following relations: for every fiber $F \in \mathcal F_{B,A}$, one has
\[
\sum_{f\in F} c_f =0.
\]

Since $|A| = |B| + 1$, every embedding $f: B \to A$ has a unique hole $x_f \in A \setminus f(B)$. By Proposition \ref{subgraph}, every connected component of $\Gamma_{B,A}$ is isomorphic to a connected component of $\Sigma_A$. Since $\Sigma_A$ as a subgraph of the underlying tree of $A$ is a forest, every connected component of $\Sigma_A$ is a tree. Fix a connected component $\Upsilon \subseteq \Gamma_{B,A}$. We first show that the subsystem of relations supported on $\Upsilon$ contributes either rank $0$ or rank $1$.

\medskip
\textbf{Case 1: $\Gamma$ contains a singleton fiber.} Let $f \in \Upsilon$ be an embedding contained in a singleton fiber $F= \{f\}$. Then the fiber relation gives $c_f = 0$. Now let $g \in \Upsilon$ be any other vertex. Since $\Gamma$ is connected, there exists a path
\[
f = f_0 \sim f_1 \sim \cdots \sim f_n = g
\]
in $\Upsilon$. Each edge corresponds to a fiber of size two, so the corresponded relation has the form
\[
c_{f_i} + c_{f_{i+1}} = 0.
\]
Since $c_f=0$, induction along the path gives $c_g = 0$. Thus all coefficients on $\Upsilon$ vanish, so $\Upsilon$ contributes rank $0$.

\medskip
\textbf{Case 2: $\Upsilon$ contains no singleton fiber.} Then every fiber meeting $\Upsilon$ is completely contained in $\Upsilon$, and has exactly two elements. In this case every edge of $\Upsilon$ determines a relation
\[
c_f + c_g = 0.
\]
Fix a vertex $f_0 \in \Upsilon$. Since $\Upsilon$ is a tree, for every vertex $g \in \Upsilon$ there is a unique path
\[
f_0 \sim f_1 \sim \cdots \sim f_n = g.
\]
The relations along this path uniquely determine
\[
c_g=(-1)^n c_{f_0}.
\]
Because $\Upsilon$ is a tree, the parity of the distance from $f_0$ to a vertex is well-defined. Hence the above formula defines a unique coefficient on every vertex of $\Upsilon$. For every edge $f_i \sim f_{i+1}$ one has $c_{f_{i+1}} = -c_{f_i}$, so all fiber relations are satisfied. Choosing $c_{f_0} = 1$ yields a nonzero solution. Thus the solution space supported on $\Upsilon$ is a free $k$-module of rank 1.

\medskip
Finally, fiber relations never mix different connected components of $\Gamma_{B,A}$. Hence the total solution space decomposes as a direct sum over connected components:
\[
\Hom_{k\D} (\Delta_A, \Delta_B) \cong \bigoplus_{\Upsilon} V_\Upsilon,
\]
where $V_\Upsilon$ is the solution space associated to $\Upsilon$. By the preceding discussion, each component contributes rank $1$ precisely when it contains no singleton fiber, and contributes rank $0$ otherwise. The conclusion then follows.
\end{proof}

Proposition \ref{subgraph} provides us a useful criterion to determine whether a connected component of $\Gamma_{B, A}$ has an embedding contained in a single fiber.

\begin{corollary} \label{frozen criterion}
Let $A,B \in \Ob(\D)$ with $|A|=|B|+1$, and let $\Upsilon$ be a connected component of $\Sigma_A$. Suppose that there exists an embedding $f: B \hookrightarrow A$ whose unique hole belongs to $\Upsilon$. Then a connected
component of $\Gamma_{B,A}$ corresponding to $\Upsilon$ contains no singleton fiber if and only if no vertex of $\Upsilon$ is adjacent to a non-mutable vertex of $A$ by an admissible edge.
\end{corollary}

\begin{proof}
By Proposition \ref{subgraph}, $\Upsilon$ is canonically identified with a connected component of $\Gamma_{B,A}$. Take $x \in \Upsilon$, and let $f: B \hookrightarrow A$ be the corresponded embedding, so that $x$ is the unique hole of $f$.

Suppose first that $x$ is adjacent to a non-mutable vertex $y$ in $A$ by an admissible edge. Let
\[
\pi: A \twoheadrightarrow A/\{x,y\}
\]
be the corresponded admissible elementary contraction. Since $y$ is non-mutable, there is no embedding $g: B \hookrightarrow A$ distinct from $f$ which lies in the same fiber of $\pi^{\ast}$ as $f$: otherwise, $y$ shall be the hole of $g$ by (3) of Lemma \ref{fibers}, and hence is mutable. Hence the fiber of $\pi^\ast$ containing $f$ is singleton. Consequently, the connected component of $\Gamma_{B,A}$ corresponded to $\Upsilon$ contains a singleton fiber.

Conversely, suppose that the connected component of $\Gamma_{B,A}$ corresponded to $\Upsilon$ contains a singleton fiber. Let $f: B \hookrightarrow A$ be an embedding contained in such a fiber, and let $x$ be the hole of $f$. Since the fiber is induced by an admissible elementary contraction, there exists an admissible edge $\{x, y\}$ in $A$ such that the fiber of $\pi^{\ast}$ induced by $\pi: A \twoheadrightarrow A/\{x, y\}$ containing $f$ is singleton. If $y$ were mutable, this admissible edge is contained in $\Upsilon$. By Proposition \ref{subgraph}, this means exactly that the fiber of $\pi^{\ast}$ containing $f$ has two embeddings corresponded to $x$ and $y$ respectively, a contradiction. Thus $y$ must be non-mutable, so $x$ is adjacent to a non-mutable vertex by an admissible edge.
\end{proof}

\begin{corollary} \label{covering}
Let $\Lambda$ be a connected component of $\Sigma_A$. Suppose that there exists an embedding $f: B \to A$ such that the connected component of $\Sigma_A$ containing the hole $x_f$ is $\Lambda$. Then exactly $|G_B|$ connected components of $\Gamma_{B,A}$ are isomorphic to $\Lambda$, where $G_B = \Aut_{\D}(B)$.
\end{corollary}

\begin{proof}
For every automorphism $\sigma\in G_B$, the embedding $f \circ \sigma$ has the same image as $f$, and hence $x_{f \circ \sigma} = x_f = x$. We claim that the embeddings $f\circ \sigma$ belong to pairwise distinct connected components of $\Gamma_{B,A}$. Indeed, suppose that $f \circ \sigma_1$ and $f \circ \sigma_2$ lie in the same connected component. Since they have the same hole $x$, Lemma \ref{multiplicity one} implies $f \circ \sigma_1 = f \circ \sigma_2$. Because $f$ is injective, it follows that $\sigma_1 = \sigma_2$. Thus distinct automorphisms yield distinct connected components.

Conversely, let $g: B \hookrightarrow A$ be any embedding with $x_g = x$. Since $|A|=|B|+1$, the hole determines the image $g(B) = A \setminus\{x\} = f(B)$. Thus both $f$ and $g$ are isomorphisms from $B$ onto the same subposet of $A$. Therefore there exists a unique automorphism $\tau \in G_B$ such that $g = f \circ \tau$. Consequently, every embedding whose hole is $x$ is of the form $f \circ \sigma$ for a unique $\sigma \in G_B$.

Now let $\Upsilon$ be any connected component of $\Gamma_{B,A}$ mapping onto $\Lambda$. By Proposition \ref{subgraph}, the map $g \mapsto x_g$ restricts to an isomorphism from $\Upsilon$ onto $\Lambda$. In particular, there exists a unique vertex $g \in \Upsilon$ satisfying $x_g = x$. By the previous paragraph, $g = f \circ \sigma$ for a unique $\sigma \in G_B$. Since distinct automorphisms give vertices lying in distinct connected components, the assignment $\Upsilon \mapsto \sigma$ is a bijection between the set of connected components of $\Gamma_{B,A}$ lying over $\Lambda$ and $G_B$.
\end{proof}

We give a few examples to illustrate the application of the above results.

\begin{example}
Let $\D$ the simplex category. As shown in Example \ref{simplex category}, for every finite chain $A$, $\Sigma_A$ is isomorphic to the underlying tree of $A$, which is connected. Consequently, when $|A|=|B|+1$, the fiber graph $\Gamma_{B,A}$ is also isomorphic to $\Sigma_A$ since it has only one connected component and $|G_B| = 1$. Thus $\Hom_{k\D}(\Delta_A, \Delta_B) \cong k$, and Proposition~\ref{thin 2} recovers the classical multiplicity-one phenomenon for simplicial modules.
\end{example}

\begin{example}
Let $\D$ be the same category as Example \ref{type-D category}. Then $\Sigma_A$ is isomorphic to the trunk part of $A$, namely a finite chain. Consequently, when $|A|=|B|+1$, the fiber graph $\Gamma_{B,A}$ is a disjoint union of $|G_B| = 2$ copies of $\Sigma_A$. It follows that $\Hom_{k\D}(\Delta_A, \Delta_B) \cong k^2$.
\end{example}

We are ready to obtain the following Dold-Kan correspondence.

\begin{theorem} \label{generalized D-K}
Let $k$ be a commutative ring, and define a $k$-linear category $\mathscr{E}$ whose objects are the same as $\D$ and $\mathscr{E} (B, A) = \Hom_{k\D}(\Delta_A, \Delta_B)$. Then:
\begin{enumerate}
\item $\D \Mod$ is equivalent to $\mathscr{E} \Mod$;

\item $\mathscr{E}$ is a $k$-linear category such that $\mathscr{E}(A,A) \cong kG_A$ and
\[
\mathscr{E}(B, A) \neq 0 \;\Longrightarrow\; |A| = |B| \ \text{or} \ |A| = |B| + 1;
\]

\item $\mathscr{E}(B, A)$ for $|A| = |B| + 1$ is a free $k$-module whose rank equals the number of connected components of $\Gamma_{B, A}$ containing no singleton fiber.
\end{enumerate}
\end{theorem}

\begin{proof}
By Proposition~\ref{Reedy of D}, each standard module is projective. Since the negative morphisms are surjective and hence have free codomain-automorphism actions, the standard-filtration argument of Corollary \ref{tree projective generators} applies to $\D$. Thus the family $\{\Delta_A\}_{A\in\Ob(\D)}$ consists of projective generators, and the equivalence follows from Morita theory for $k$-linear categories with several objects.

(2) This follows from Propositions \ref{semi-orthogonal} and \ref{thin 1}.

(3) This follows from Proposition \ref{thin 2}.
\end{proof}

We describe an immediate consequence of this theorem.

\begin{corollary}
If $k$ is noetherian, then every finitely generated $\D$-module is noetherian. A similar conclusion holds in the artinian case.
\end{corollary}

\begin{proof}
Note that every representable $\mathscr{E}(B, -)$ is a free $k$-module of finite rank. Indeed, by the previous theorem, $\mathscr{E}(B, A) \neq 0$ only if $A=B$ or $|A|=|B|+1$, and there are only finitely many such $A$. Thus the conclusion holds for representable $\mathscr{E}$-modules, and hence holds for finitely generated $\mathscr{E}$-modules. The conclusion for $\D$-modules follows from the Morita equivalence.
\end{proof}

When $k$ is a field, we have the following classification of irreducible $\D$-modules. Note that in this case, irreducible $\mathscr{E}$-modules are precisely irreducible $kG_B$-modules, viewed as $\mathscr{E}$-modules supported on $B$, for all $B \in \Ob(\mathscr{E})$.

\begin{corollary}
Let $k$ be a field and $B$ be an object of $\D$, and define
\[
R_B = \sum_{\substack{|A| = |B| + 1\\ f: \Delta_A \to \Delta_B}} \operatorname{Im}(f).
\]
Then
\[
\Delta_B/ \bigl(R_B + \operatorname{rad}(kG_B) \Delta_B \bigr)
\]
is a semisimple $\D$-module. Moreover, its simple summands are precisely the irreducible $\D$-modules corresponded to the simple $kG_B$-modules.
\end{corollary}

\begin{proof}
Let $\mathscr{E}$ be the category in Theorem \ref{generalized D-K}, and let $F: \D \Mod \to \mathscr{E} \Mod$ be the equivalence. Clearly, the module $\Delta_B$ corresponds to the representable projective module $Q_B = \mathscr E(B, -)$.

Let $J$ be the ideal of $\mathscr{E}$ generated by all morphisms $\mathscr{E}(B,A)$ with $|A| = |B|+1$. By Theorem \ref{generalized D-K}, the quotient category $\mathscr{E}/J$ has only endomorphisms, and
\[
(\mathscr{E}/J)(B,B) \cong kG_B .
\]
Hence the quotient $Q_B/JQ_B$ is naturally identified with the regular $kG_B$-module.

For a morphism $\alpha \in \mathscr{E} (B,A)$, Yoneda's lemma gives a homomorphism
\[
Q_A = \mathscr{E}(A,-) \longrightarrow Q_B = \mathscr{E}(B,-).
\]
Therefore $JQ_B$ is generated by the images of these maps for all $|A| = |B|+1$. Transporting this description through the equivalence $F$, we obtain
\[
F^{-1}(JQ_B)= \sum_{\substack{|A| = |B|+1\\ \alpha: \Delta_A \to \Delta_B}} \Ima(\alpha) = R_B.
\]
Since $Q_B/JQ_B$ is a $kG_B$-module, its quotient
\[
(Q_B/JQ_B) \Big/ \operatorname{rad}(kG_B)(Q_B/JQ_B)
\]
is semisimple. Applying the inverse equivalence $F^{-1}$ yields
\[
\Delta_B/ \bigl( R_B + \operatorname{rad}(kG_B)\Delta_B \bigr).
\]
Since equivalences preserve semisimplicity, this module is semisimple.

Finally, the semisimple top of the regular $kG_B$-module is
\[
(Q_B/JQ_B) \Big/ \operatorname{rad}(kG_B)(Q_B/JQ_B).
\]
Hence its simple summands are in bijection with simple $kG_B$-modules. Transporting them back through $F$ gives precisely the irreducible $\D$-modules corresponded to simple $kG_B$-modules.
\end{proof}

We conclude this subsection with two related categories of rooted trees which do not necessarily satisfy Condition (C2). Nevertheless, their standard modules form pairwise Hom-orthogonal families of projective generators, yielding stronger decomposition equivalences.

A \emph{spider} is a finite rooted poset obtained as follows: start with a finite collection of finite chains, and identify all their minimal elements to a single root. Equivalently, a spider is a rooted tree in which every vertex distinct from the root has at most one immediate successor. We allow the singleton tree and the spiders with one leg.

Let $\mathscr S$ be the category of all finite spiders whose morphisms are generated by root-preserving embeddings and admissible elementary contractions. Every edge of a spider is admissible, and contracting any edge again yields a spider. In particular, the unique minimal object is the singleton tree. The root is never flexible, whereas every other vertex is flexible. Consequently, for a spider with more than one leg, the flexible vertices are not connected by admissible edges. Thus this root-preserving category does not in general satisfy (C2). However, we can still prove its Reedy structure and the orthogonal decomposition of its representation category.

\begin{corollary} \label{spider}
The above category $\mathscr S$ of spiders is a generalized Reedy category. Moreover, the family $\{\Delta_A\}_{A \in \Ob(\mathscr S)}$ forms an orthogonal family of projective generators of $\mathscr S \Mod$. Consequently,
\[
\mathscr{S} \Mod \simeq \prod_{A \in \Ob(\mathscr S)} kG_A \Mod,
\]
where $G_A = \Aut_{\mathscr S}(A)$.
\end{corollary}

\begin{proof}
Firstly, we show that $\mathscr{S}$ has a generalized Reedy structure. Let $\iota: B \hookrightarrow A$ be a root-preserving embedding and let $\pi: A \to A/e$ be an admissible elementary contraction. Since spiders are closed under such contractions, by Lemma \ref{exchange}, one has $\pi \circ \iota = \bar{\iota} \circ \pi_B$, where $\pi_B$ is an identity or an admissible elementary contraction and $\bar{\iota}$ is an embedding. Both $\pi \circ \iota$ and $\pi_B$ preserve the roots, so does $\bar{\iota}$. Consequently, this factorization is contained in $\mathscr{S}$. By repeating this exchange, we get a desired Reedy factorization for every morphism in $\mathscr{S}$. The uniqueness of this factorization is inherited from that of $\T$. Thus $\mathscr S$ is a generalized Reedy category.

Since every admissible contraction of a spider has a root-preserving sections, the normalization idempotent $\Psi_A$ of Subsection \ref{D-K correspondence} belongs to $k\mathscr S(A,A)$. By mimicking its proof, we get
\[
P_A^{\mathscr S} = P_A^{\mathscr S} \Psi_A \oplus \mathfrak I_A^{\mathscr S}, \qquad \Delta_A^{\mathscr S} \cong P_A^{\mathscr S}\Psi_A.
\]
Thus every standard module is projective. By the standard-filtration argument, they form a family of projective generators.

It remains to prove that the standard modules are pairwise Hom-orthogonal. Fix spiders $A,B$ with $|A|>|B|$ and a root-preserving embedding $f:B\hookrightarrow A$. Let $v$ be a vertex of minimum depth in $A\setminus f(B)$. Since $f$ preserves the root, $v$ is not the root, and its parent $p$ belongs to $f(B)$. We have two cases:
\begin{itemize}
\item If $p$ is the root, let $\pi$ contract the edge $p \lessdot v$. Then $\pi \circ f$ is a root-preserving embedding, and the fiber containing $f$ is singleton since any different embedding in that fiber would have to move the image of the root from $p$ to $v$, contrary to root preservation. Thus the fiber relation gives $c_f = 0$.

\item If $p$ is not the root, it has $v$ as its unique immediate successor. Take $u \in B$ with $f(u) = p$, and define $g: B \hookrightarrow A$ by sending $u$ to $v$ and agreeing with $f$ elsewhere. Since every vertex of $A$ other than $p,v$ has the same comparability relation with $p$ and $v$, the map $g$ is again a root-preserving embedding. Contracting $p \lessdot v$ gives $\pi \circ f = \pi \circ g$, so the corresponding two-element fiber yields $c_f + c_g = 0$. The hole $v$ of $f$ has been replaced by the hole $p$ of $g$, of strictly smaller depth. Iterating this argument eventually reaches a singleton fiber, and hence $c_f = 0$.
\end{itemize}
Since $f$ was arbitrary, it follows that every compatible family of coefficients vanishes. By Theorem \ref{hom spaces general}, we deduce that
\[
\Hom_{k\mathscr S}(\Delta_A, \Delta_B) = 0
\]
when $|A|>|B|$. The other off-diagonal Hom-spaces also vanish by Proposition \ref{semi-orthogonal}. In degree zero, only isomorphic objects admit positive morphisms, and $\operatorname{End}_{k\mathscr S}(\Delta_A)\cong kG_A$. Thus the standard modules are pairwise Hom-orthogonal, and the equivalence follows from Morita theory.
\end{proof}

\begin{remark}
For $n\geqslant 0$, let $A_n = \{0,1,\ldots,n\}$ be the star-type poset with $0 \leqslant i$ for $1 \leqslant i \leqslant n$; in particular, $A_0$ is the singleton tree. Let $\mathscr S_{\mathrm{st}}$ be the category of these posets whose morphisms are generated by root-preserving embeddings and admissible elementary contractions. Then $\Aut(A_n)\cong S_n$, where $S_0$ is the trivial group. Every admissible elementary contraction identifies the root with a leaf, while root-preserving embeddings send distinct leaves to distinct leaves. Consequently, a morphism $f: A_m \to A_n$ is uniquely determined by its restriction to the leaves, where each leaf is either mapped
to the root or to a leaf of $A_n$, and distinct leaves mapped away from the root have distinct images. Conversely, every such map is obtained by contracting the leaves mapped to the root, followed by a root-preserving
embedding. Thus morphisms $A_m\to A_n$ are naturally identified with partial injections $\{1, \ldots, m\} \dashrightarrow \{1,\ldots,n\}$. It follows that $\mathscr S_{\mathrm{st}}$ is equivalent to $\FI\sharp$, the category of finite sets (including the empty set) and partial injections.

Since stars are precisely spiders whose legs all have length one, $\mathscr S_{\mathrm{st}}$ is a full subcategory of $\mathscr S$. Moreover, it is closed under admissible elementary contractions, and their canonical sections are root-preserving embeddings. Hence the normalization construction and the singleton-fiber argument for $\mathscr S$ also apply to $\mathscr S_{\mathrm{st}}$. Consequently, the standard modules for $\FI\sharp$ form a pairwise Hom-orthogonal family of projective generators, recovering the classical equivalence
\[
\FI\sharp \Mod \simeq \prod_{n \geqslant 0} kS_n \Mod
\]
established in \cite{CEF}.

The preceding corollary extends this decomposition from stars to all finite spiders, with the symmetric groups replaced by their automorphism groups. Explicitly, if a spider $A$ has $m_i$ legs of length $i$, then
\[
G_A \cong \prod_{i \geqslant 1}S_{m_i}.
\]
\end{remark}

\begin{remark}
Applying Theorem~\ref{decomposition} over a characteristic-zero field to the preceding corollary gives the numerical symmetry
\[
|\mathscr S^+(A,B)| = |\mathscr S^-(B,A)|,
\]
which is not at all apparent from the combinatorics of spiders. Indeed, embeddings in $\mathscr S^+(A, B)$ are obtained by adding vertices to the legs of a spider, whereas morphisms in $\mathscr S^-(B, A)$ are composites of admissible elementary contractions. Thus the two sets arise from rather different combinatorial constructions. However, Theorem \ref{decomposition} shows that this equality has a natural representation-theoretic meaning. Together with the orthogonality of the standard modules, it is precisely the condition required for the equivalence stated in the corollary.

For the spider category, the argument goes in the opposite direction: we first establish the orthogonal decomposition through normalization and the singleton-fiber relations, and then recover the cardinality symmetry from Theorem~\ref{decomposition}. Thus a representation-theoretic equivalence explains a non-obvious identity between two combinatorially different classes of morphisms.
\end{remark}

\end{document}